\documentclass[a4paper,12pt]{amsart}

\usepackage[utf8]{inputenc}

\usepackage{amsmath,amssymb,amsxtra,comment,graphicx,psfrag}
\usepackage{bm,mathrsfs}
\usepackage{mathtools}
\usepackage{xspace}
\usepackage{stmaryrd}
\usepackage{color}
\usepackage{lineno}
\newcommand{\tred}{\textcolor{black}}

\newcommand{\beginsupplement}{%
        \setcounter{table}{0}
        \renewcommand{\thetable}{S\arabic{table}}%
        \setcounter{figure}{0}
        \renewcommand{\thefigure}{S\arabic{figure}}%
        \setcounter{equation}{0}
        \renewcommand{\theequation}{S\arabic{equation}}%
     }

\usepackage{enumerate,comment,graphicx,psfrag}

\usepackage{hhline}

\usepackage{bookmark}
\usepackage{nameref}

\usepackage{verbatim}
\usepackage{tabls}

\usepackage{fancyhdr}
\usepackage{epsfig}
\usepackage{epsfig,subfigure,epstopdf}

\usepackage{pstricks,pst-plot,pst-func}
\usepackage{pspicture}
\usepackage{curves}

\usepackage{bbm}

\include{rgb}

\def\div{\text{\rm div}}

\def\nst2{\| _*} 
 
\def\a12{A_h ^{1/2} } 
\def\d{{\mathrm d}}
\def\tr|{|\! |\! |}

\def\R {{\mathbb R}}

\def\E{{\mathcal{E}}}

\def \a{\alpha }

\def\T_h{{{\mathcal T}_h}}

\def\<{{\langle }}
\def\>{{\rangle }}

\def\S{{\mathcal{S}}}

\newcommand{\mbph}{\bm \phi}

\DeclarePairedDelimiter{\norm}{\lVert}{\rVert}

\DeclareSymbolFont{matha}{OML}{txmi}{m}{it}
\DeclareMathSymbol{\varv}{\mathbf}{matha}{118}

\theoremstyle{definition}
\newtheorem{prop}{Proposition}[section]

\newtheorem{Theorem}{Theorem}[section]

\newtheorem{Remark}{Remark}[section]

\newcommand\restr[2]{{
  \left.\kern-\nulldelimiterspace 
  #1 
  \vphantom{\big|} 
  \right|_{#2} 
  }}

\NewDocumentCommand{\dgal}{sO{}m}{%
  \IfBooleanTF{#1}
    {\dgalext{#3}}
    {\dgalx[#2]{#3}}%
}

\NewDocumentCommand{\dgalext}{m}{%
  \sbox0{%
    \mathsurround=0pt 
    $\left\{\vphantom{#1}\right.\kern-\nulldelimiterspace$%
  }%
  \sbox2{\{}%
  \ifdim\ht0=\ht2
    \{\kern-.625\wd2 \{#1\}\kern-.625\wd2 \}%
  \else
    \left\{\kern-.7\wd0\left\{#1\right\}\kern-.7\wd0\right\}%
  \fi
}

\NewDocumentCommand{\dgalx}{om}{%
  \sbox0{\mathsurround=0pt$#1\{$}%
  \sbox2{\{}%
  \ifdim\ht0=\ht2
    \{\kern-.625\wd2 \{#2\}\kern-.625\wd2 \}%
  \else
    \mathopen{#1\{\kern-.7\wd0 #1\{}
    #2
    \mathclose{#1\}\kern-.7\wd0 #1\}}
  \fi
}

\renewcommand{\R}{{\mathbb{R}}}
\newcommand{\dx}{\ \mathrm{d}x}

\renewcommand{\T}{\mathbb{T}^d}

\begin{document}
\theoremstyle{plain}
\newtheorem{theorem}{Theorem}[section]
\newtheorem{lemma}{Lemma}[section]
\newtheorem{proposition}{Proposition}[section]
\newtheorem{corollary}{Corollary}[section]
\newtheorem{problem}{Problem}[section]

\theoremstyle{definition}
\newtheorem{definition}[theorem]{Definition}

\newtheorem{example}[theorem]{Example}

\newtheorem{remark}{Remark}[section]
\newtheorem{remarks}[remark]{Remarks}
\newtheorem{note}{Note}
\newtheorem{case}{Case}

\numberwithin{equation}{section}
\numberwithin{table}{section}
\numberwithin{figure}{section}




\title[Deep Ritz - Finite Element methods]
{Deep Ritz - Finite Element methods: Neural Network Methods trained with Finite Elements}

\author[Georgios Grekas]{Georgios Grekas}
\address{
Institute of Applied and Computational Mathematics, 
FORTH, 700$\,$13 Heraklion, Crete, Greece /
Aerospace Engineering and Mechanics, University of Minnesota, Minneapolis, USA 
/Computer, Electrical, Mathematical Sciences \& Engineering Division, King Abdullah University of Science and Technology (KAUST), Thuwal, Saudi Arabia} 
\email {\href{mailto:grekas.g@gmail.com}{grekas.g{\it @\,}gmail.com}}

\author[Charalambos G. Makridakis]{Charalambos G. Makridakis}
\address{
DMAM, University of Crete / 
Institute of Applied and Computational Mathematics, 
FORTH, 700$\,$13 Heraklion, Crete, Greece, and MPS, University of  Sussex, Brighton BN1 9QH, United Kingdom} 
\email {\href{mailto:C.G.Makridakis@iacm.forth.gr}{C.G.Makridakis{\it @\,}iacm.forth.gr}}

\date{\today}


\begin{abstract}
While much attention of neural network methods is devoted to high-dimensional PDE problems, in this work we consider methods designed to work for elliptic problems on domains $\varOmega \subset \R ^d, $ $d=1,2,3$ in association with more standard finite elements. We suggest to connect finite elements and neural network approximations  through \emph{training}, i.e., using finite element spaces to compute the integrals appearing in the loss functionals. This approach, retains the simplicity of classical neural network methods for PDEs, uses well established finite element tools (and software) to compute the integrals involved and it gains in efficiency and accuracy. 
 We demonstrate that the proposed methods are stable and furthermore, we establish that the resulting approximations converge to the solutions of the PDE. 
 Numerical results indicating the efficiency and robustness of the proposed algorithms are presented.
\end{abstract}

\maketitle



\section{Introduction and method formulation}\label{Se:1}
While much attention in neural network methods is focused on high-dimensional PDE problems, this work explores methods designed for domains $\varOmega \subset \R^d$ with $d=1, 2, 3$, in conjunction with more traditional finite element techniques. We use finite element interpolation to train the continuous, and inherently non-computable, loss function of the original neural network method. Unlike standard neural network approaches for PDEs, which typically rely on collocation-type training (whether random or deterministic), our approach minimises over neural network spaces using specially designed loss functions that incorporate a finite element-based approximation of the continuous loss. This significantly reduces the number of back-propagation calls within the algorithm, resulting in stable and robust methods for approximating partial differential equations. In this article, we focus on linear elliptic problems, but the method can also be extended to other types of problems, linear or nonlinear. Additionally, these methods can be integrated with well-established techniques in the finite element community, such as adaptivity and mesh generation, to create hybrid algorithms that combine the strengths of both neural networks and finite element methods.

\subsection{The model problem}\label{SSe:1.1}
We start with an idealised formulation of  simple  boundary value problem   
\begin{equation}\label{ivp-AC}
\left \{
\begin{alignedat}{3}
&-\varDelta u  = f\quad &&\text{in}\,\,&&\varOmega \\[2pt]
&\quad  u =0 \quad &&\text{on}\,\,&&\partial\varOmega \\[2pt]
\end{alignedat}
\right .
\end{equation}
in a polygonal   domain
$\varOmega\subset\R^d, 1\leqslant d\leqslant 3$.
The natural   energy functional associated to this problem is 
\begin{equation}\label{en-func} 
\mathcal{E} (u)=\int_{\varOmega} \Big( \frac12 |\nabla u|^2  - f\, u \Big) \d x  \, ;
\end{equation}
then, the   solution of \eqref{ivp-AC} is the unique minimiser of the problem 
\begin{equation}\label{mm}
	\min  _ {v \in  H^1 _0 (\varOmega ) } \E (v)\, .
\end{equation}
The boundary conditions can be also imposed weakly within $\mathcal{E} .$ The deep Ritz method is based on the minimisation of the above functional on discrete neural network spaces appropriately discretised with collocation type methods to yield computable  approximations.  It will be useful to introduce an intermediate method involving 
only  minimisation of $ \E (v)$ on discrete neural network spaces without further discretisation of the functional. 
This approximation is in principle non computable, but it will be useful to introduce it in order to motivate our approach. 

\subsection{Discrete spaces generated by Neural Networks}\label{Se:1NN}
%
We follow the exposition of \cite{GGM_pinn_2023, Mishra_dispersive:2021} by considering functions 
$u_\theta$ defined through neural networks.
A deep neural network maps every point $x\in \varOmega$ to a number $u_\theta (x) \in \R$, through
\begin{equation}\label{C_L}
	u_\theta(x)= C_L  \circ \sigma  \circ C_{L-1} \cdots \circ\sigma \circ C_{1} (x) \quad \forall x\in \varOmega.
\end{equation}
%
 The process
 \begin{equation}
 \mathcal{C}_L:= C_L  \circ \sigma  \circ C_{L-1} \cdots \circ\sigma \circ C_{1} 
\end{equation}
is in principle a map $\mathcal{C}_L : \R ^m \to \R ^{m'} $; in our particular application, $m =d$ and $m'=1.$ 
The map $\mathcal{C}_L $ is a neural network with $L$ layers and activation function $\sigma.$ Notice that to 
define $u_\theta(x) $ for all $x\in \varOmega$ we use the same $\mathcal{C}_L ,$ thus $u_\theta(\cdot ) =\mathcal{C}_L  (\cdot ) .$
Any such map $\mathcal{C}_L$ is characterised by the intermediate (hidden) layers $C_k$, which are affine maps of the form 
\begin{equation}\label{C_k}
	C_k y = W_k y +b_k, \qquad \text{where }  W_k \in \R ^ {d_{k+1}\times d_k}, b_k \in \R ^ {d_{k+1}}.
\end{equation} 
Here the dimensions $d_k$ may vary with each layer $k$ and $\sigma (y)$ denotes the vector with the same number of components as $y$,
where $\sigma (y)_i= \sigma(y_i)\, .$ 
The index $\theta$ represents collectively all the parameters of the network $\mathcal{C}_L,$ namely $W_k , b_k, $ $k=1, \dots, L .$ 
The set of all networks $\mathcal{C}_L$ with a given structure (fixed $L, d_k,  k=1, \dotsc, L\,$) of the form \eqref{C_L}, \eqref{C_k}
is called $\mathcal{N}.$ The total dimension (total number of degrees of freedom) of  $\mathcal {N} ,$ is $\dim{\mathcal {N}}= \sum _{k=1} ^L d_{k+1} (d_k +1) \, .$ 
We now define the space of functions 
\begin{equation}
	V _{\mathcal{N}}= \{ u_\theta : \varOmega \to \R ,  \ \text{where }  u_\theta (x) = \mathcal{C}_L (x), \ \text{for some } \mathcal{C}_L\in  \mathcal{N}\, \} \, .
\end{equation}
It is important to observe that $V _{\mathcal{N}}$ is not a linear space. \tred{The formulation of the method and the convergence results do not depend on the specific activation functions or neural network architecture. For the analysis, it is assumed that the discrete spaces satisfy the approximation properties outlined in Section 2.1. Moreover, the convergence proof, as detailed in Theorem 2.1, becomes more intricate when the regularity of the discrete spaces is restricted to Lipschitz continuity, which is characteristic of the ReLU activation function.}
Given that  there is a one-to-one correspondence between parameters $\theta $ and functions 
\begin{equation}
	\theta \mapsto  u_\theta \in V _{\mathcal{N}}\, , 
\end{equation}
the space
\begin{equation}
	\Theta = \{ \theta \, : u_\theta \in V _{\mathcal{N}}\}.
\end{equation}  
 is a linear subspace of $ \R ^{\dim{\mathcal {N}}}.$

\subsection{Discrete minimisation  on  $V _{\mathcal{N}}$} We consider now the (theoretical) scheme:

\begin{definition}\label{abstract_mm_nn} Assume that the problem 
\begin{equation}\label{mm_nn:abstract}
	\min  _ {v \in  V _{\mathcal{N}} } \E (v)
\end{equation}
has a solution $v^\star \in V _{\mathcal{N}} .$ We call  $ v^\star \,  $ a deep-Ritz minimiser of $\E \, .$
\end{definition}
To yield a computable approximation, one observes that $\E \,  $ should be further discretised, since although derivatives of neural network functions are computable through back propagation, their integrals are not. Applying a deterministic or Monte-Carlo integration 
will yield a \emph{fully discrete} method. This is the key idea of the Deep Ritz method, \cite{e2017deep}.  The  discretisation of the functional
 $\E \,  $ is typically called in the literature \emph{Training,} since it parallels the training through data step of neural network algorithms, although for solving PDEs we do not have always available data to be used.  
 
 A subtle issue arises with boundary conditions because the non-local nature of neural network approximations makes it challenging to enforce constraints such as $  V _{\mathcal{N}} \subset H^1_0 (\varOmega).$ Typically, \eqref{mm_nn:abstract}
must be modified to include a term like
$$ \int _ {\partial \varOmega } \, |v|^2 \, \dx \, , $$
thereby weakly enforcing zero boundary conditions. One advantage of our approach is that it offers a clear solution to this issue, as discussed in Section \ref{bc}. However, for simplicity, we omit boundary losses in the presentation of the methods below; our method with boundary losses is described in Section \ref{bc}.


\subsection{Training approaches.}
\label{training_approaches}
 Computable discrete versions of the energy $\E(u_\theta)$ through training  
 can be achieved through different ways. We first describe the known methods based on quadrature/collocation and then we discuss the  method 
 suggested in the present work based on finite elements. 

 
 \subsubsection{Training through quadrature/collocation} One uses appropriate quadrature for integrals over $\varOmega.$ Such a quadrature requires 
 a set $K_h$ of discrete points $z\in K_h$  and corresponding nonnegative weights $w_z$
 such that 
 \begin{equation}
\label{quadrature}
\sum _{z\in K_h} \, w_z \, g(z) \approx \int _{\varOmega} \, g(z) \, \d x .
\end{equation}
  With the help of \eqref{quadrature} we define
%
 %
\begin{equation}
\label{E_h}
\mathcal{E}_{Q, h}( g )  = \sum _{z\in K_h} \, w_z \,\Big (  \frac 12 |\nabla g(z)| ^2  - f(z) g(z)\,\Big ) \, . 
\end{equation}

\begin{definition}
Assume that the problem 
\begin{equation}\label{ieE-minimize_NNQ}
	\min  _ {v \in  V _{\mathcal{N}} } \E _ {Q, h}(v)
\end{equation}
has a solution $v^\star \in V _{\mathcal{N}} .$ We call  $ v^\star \,  $ a Q-deep-Ritz minimiser of $ \E _ {Q, h} \, .$
\end{definition}
The deterministic quadrature \eqref{quadrature} naturally results in the discrete energy \eqref{E_h}, which necessitates evaluating $\frac 12 |\nabla g(z)| ^2  - f(z) g(z)$ at the quadrature points. This can be computationally accomplished using the back-propagation algorithm implemented in standard neural network software packages.


 \subsubsection{Training through Monte-Carlo quadrature/collocation} 
The  formulation in  \eqref{mm_nn:abstract} is quite flexible and allows probabilistic quadrature as well. In fact, we may  consider a collection $ X_1,X_2,\ldots$ of i.i.d.\ $\varOmega$-valued random variables, defined on an appropriate probability space 
representing a random choice of points in $\varOmega\, .$
Monte Carlo integration is formulated   in a suitable probabilistic framework: 
let $\omega $   be a fixed instance, and $X_i(\omega )\in \varOmega$ the corresponding values of the random variables. Consider the discrete energy, 
\begin{equation}\label{prob_E}
\mathcal{E}_{N, \omega }( g )  =  \frac 1N  \sum _{i=1 } ^N    \, \frac 12 |\nabla g(X_i(\omega)\, )| ^2  - f(X_i(\omega)\, ) g(X_i(\omega)\, )    
 \end{equation}
 The discrete minimisation problem for each instance is 
\begin{equation}\label{ieE-minimize_NNQ}
	\min  _ {v \in  V _{\mathcal{N}} } \E _{N, \omega }(v)\, . 
\end{equation}
We expect that for sufficiently large number of samples $N,$ this energy will approximate $\mathcal{E}(v).$  
Such methods are more appropriate in higher dimensions and are among (along with quasi-MC methods) the most popular training approaches for neural network discretisation of PDEs. In low dimensions, however, are computationally quite demanding. 

\subsubsection{Training through Finite Elements}
To introduce our method we shall need some standard finite element notation and terminology, cf. e.g., \cite{BrennerSFE}. Let $T_h$ be a shape regular triangulation of a polygonal domain $\varOmega$ with mesh size $h=h(x).$
For $K\in T_h$, $K$ is an element of the triangulation, and  $\mathbb{P}_q(K)$ denotes the set of polynomials of degree less or equal to $q$. 
We define the standard space of continuous  piecewise polynomial functions as  
\begin{equation}\label{Sh_def}
\tilde \S_h(\varOmega):=\left\{v\in C^0(\bar \varOmega)\,:\,\restr{v}{K}  \in \mathbb{P}_q(K), K \in T_h \right\}.
\end{equation}
We also consider the finite element space where zero boundary conditions are enforced
\begin{equation}\label{Sh_def}
 \S_h(\varOmega):=\left\{v\in C_0^0(\bar \varOmega)\,:\,\restr{v}{K}  \in \mathbb{P}_q(K), K \in T_h \right\}.
\end{equation}

Without loss of generality we consider Lagrangian elements:
  Let $\{ \Phi _z\} _ {z\in \mathcal {Z}}, $ be  the Lagrangian basis of $\S _{h},$
where $\mathcal {Z} $ denotes the set of degrees of freedom. 
The interpolant 
$I_{\S_h} :   C^0(\bar \varOmega)  \rightarrow \S_h(\varOmega)$ is defined as 
\begin{align}
I_{\S_h} u (x)= \sum_{z\in \mathcal {Z} }  u(z) \Phi_z(x), \text{ for } u\in  C^0(\bar \varOmega)\, . 
\end{align}
It is clear now that an   approximation of 
$\E $ is provided by 
\begin{equation}
\label{E_hFE}
\mathcal{E}_{\S _h}( g )  =\int_{\varOmega} \Big( \frac12 |\nabla I_{\S _h} ( g)|^2  -  I_{\S _h} ( f\, g) \Big) \d x  \, . 
\end{equation}
%
%
%
%
 %
%
The \emph{finite element-deep Ritz} method is then defined by minimising this discrete energy over the same neural network space 
$ V _{\mathcal{N}}:$

\begin{definition}
Assume that the problem 
\begin{equation}\label{ieE-minimize_NNQ}
	\min_ {v \in  V _{\mathcal{N}} }\mathcal{E}_{\S _h}(v)
\end{equation}
has a solution $v^\star \in V _{\mathcal{N}} .$ We call  $ v^\star \,  $ a $\S _h$-finite element-deep Ritz minimiser of $\mathcal{E}_{\S _h}\, .$
\end{definition}
Some clarifications are in order: The finite element space and the corresponding interpolant
 $I_{\S _h}$ are utilised solely to define the discrete energy. Specifically, for $g\in   V _{\mathcal{N}}, $ the function  $I_{\S _h} g \in \S _h .$ Moreover, both  $I_{\S _h} g $ and its gradient 
$\nabla I_{\S _h} g\, ,$ are piecewise polynomial functions involving only point values of $g$ at the Lagrangian degrees of freedom. Thus the corresponding integrals are readily computable using standard finite element tools, which highlights a key advantage of the method.  
The computation of $\int_{\varOmega}   |\nabla I_{\S _h} ( g)|^2 dx$ is straightforward and does not require the backpropagation algorithm. However, backpropagation is still employed to calculate the derivative of $\mathcal{E}_{\S _h} ,$
 which is necessary for the iterative approximation of minimisers.

\subsection{Boundary conditions.}\label {bc} Our approach permits the application \tred{ of finite element type   methods to weakly} impose homogeneous (or more general) boundary conditions. Using finite element training in the discrete functional, among other advantages, one is able to use the standard toolbox associated  to piecewise polynomials defined on triangulations. Therefore, e.g., Nitche's method to treat the boundary conditions 
can be made precise as in the standard finite elements, \tred{\cite{nitsche1971variationsprinzip}.} This was challenging  through other training methods, as it was not possible to include ``balanced" boundary terms in the functional.   \tred{ For interesting applications of   Nitsche's method  in the neural network setting with probabilistic training, see \cite{ming2021deep} and also \cite{georgoulis2023discrete} where a detailed discussion on the choice of the penalty parameter of the method is provided. }

Specifically, we seek minimizers of the discretized problem for $v \in V_{\mathcal{N}}$ satisfying the boundary conditions $v =g_0$ on $\partial \varOmega$ imposed weakly within the discrete functional.
Let $E_h^b$ denote the set of the boundary edges from the triangulation $T_h$. Then, Dirichlet boundary conditions are applied through 
Nitsche's method by adding to the minimisation problem the energy term 
\begin{align}\label{Nitsche_bcs}
  \sum_{e \in E_h^b}\frac{\alpha}{h_e} \int_e |v-g_0|^2 ds, \text{ for some } \alpha >0,
\end{align}
where $h_e$ is the diameter of the boundary edge $e $ of the decomposition. 
Notice that     $g_0=0$ in the  case of Dirichlet boundary conditions.
Thus when training with finite elements is considered the discrete energy takes the form
\begin{equation}
\label{E_hFE_nitsche}
\mathcal{E}_{\tilde \S _h, wb}( g )  =\int_{\varOmega} \Big( \frac12 |\nabla I_{\tilde \S _h} ( g)|^2  -  I_{\tilde \S _h} ( f\, g) \Big) \d x  + \sum_{e \in E_h^b}\frac{\alpha}{h_e} \int_e |I_{\tilde \S _h} ( g) -g_0|^2 ds
\end{equation}
for some $ \alpha >0 $ being a penalty parameter. \tred{In our case, the penalty parameter depends on the finite element spaces considered, and
the constants appearing in the required inverse inequalities. As \tred{is} well known in the finite element literature, the factor $\frac{1}{h_e} $ is crucial to balance the boundary discrete norms to  the $H^1-$ semi-norm 
at $\varOmega\, .$ 
For further details on the lower bounds for the penalty parameter $\alpha$, please refer to \cite[Chapter 37]{ern2021finite}}.

The corresponding minimisation problem is 
\begin{equation}\label{ieE-minimize_NN_wbFE}
	\min_ {v \in  V _{\mathcal{N}} }{E}_{\tilde \S _h, wb}(v)\, .
\end{equation}

%
%
%
\subsection{General Elliptic Problems}\label{gen_ell}
Adopting standard finite element quadrature approaches finite element training  can be applied to general elliptic problems of the form,
\begin{equation}\label{ellipticPDE}
\begin{alignedat}{3}
&L\, u = f\quad &&\text{in}\,\,&&\varOmega \\[2pt]
\end{alignedat}
\end{equation}
with boundary conditions $u =0\ \text{on}\  \partial\varOmega$. Here $ u : \varOmega \subset \mathbb{R}^d \rightarrow \mathbb{R} , \: \varOmega $ is an open, bounded set with smooth enough boundary, 
$f \in L^2 (\varOmega)  $ and $ L $ a self-adjoint elliptic operator of the form
\begin{equation}\label{EllipticOperator}
\begin{gathered}
Lu := - \sum_{1 \leq i,j \leq d} \big ( a_{ij} u_{x_i } \big )_{  x_j}  +cu \\
\textrm{where} \;\: \sum_{i,j} a_{ij}(x) \xi_i \xi_j \geq \theta | \xi|^2 \;\; \textrm{for any} \;\: x \in \varOmega \;\: \textrm{and any} \;\: \xi \in \mathbb{R}^n, \;\;\; \textrm{for some} \;\: \theta >0\, .
\end{gathered}
\end{equation}
Also, the coefficients are  smooth enough satisfying   $ a_{ij} = a_{ji} $  and $c\geq c_0> 0.$ 
 The methods and analysis herein can be extended to other boundary conditions with appropriate  modifications.  
We shall need the bilinear  form  $ B : H^1_0(\varOmega) \times H^1_0(\varOmega) \rightarrow \mathbb{R} ,$ defined by 
\begin{equation}\label{BilinearForm}
\begin{gathered}
B(u,v) = \int_\varOmega \Big (\, \sum_{i,j=1}^n a_{ij}u_{x_i}v_{x_j}  +cuv \: \Big ) \, \d x \, .
\end{gathered}
\end{equation}
The analog of the 
Dirichlet energy in this case is 
  \begin{equation}\label{en-func_genEll} 
\mathcal{E} (u)=  \frac12 B(u,u)  - \int_{\varOmega}   f\, u \, \d x  \, .
\end{equation}
Following \cite[Chapter IV, Sections 26, 28] {CiarletFE} we assume that finite element quadrature can be applied 
on  $ B : \S _h \times \S _h \rightarrow \mathbb{R} ,$ yielding a discrete bilinear form 
$ B _h : \S _h \times \S _h \rightarrow \mathbb{R} , $ for which the following two properties are satisfied
 \begin{equation}\label{genEll_quadr_ass} 
 \begin{split}
 & \tilde \alpha \| v_h \| ^ 2 _{H^1 (\varOmega )} \leq \, B_h (v_h, v_h )\, , \qquad \tilde  \alpha >0\, , \\
 &	\lim _{h \to 0 } \, \sup _{w_h\in \S _h} \frac {\big | B (I_{\S _h} ( g), w_h ) - B_h (I_{\S _h} ( g), w_h )\big |}
 {\| w_h \|  _{H^1 (\varOmega )} }=0\, ,
 \end{split}
\end{equation}
for sufficiently smooth $g\in H^1 _0 (\varOmega )\, $ (\tred{for more precise estimates depending on the regularity of $g,$ see e.g., 
the proof of   \cite[ Theorem 29.1] {CiarletFE}.)}
Then  the analog of the energy trained with finite elements 
is 
\begin{equation}
\label{E_hFE_genEll}
\mathcal{E}_{\S _h}( g )  = \frac12 B_h (I_{\S _h} ( g),I_{\S _h} ( g) ) -\int_{\varOmega}    I_{\S _h} ( f\, g)  \d x  \, . 
\end{equation}
In \cite[Chapter IV] {CiarletFE} a detailed finite element analysis with particular examples of quadrature rules satisfying these properties is provided. The deep Ritz finite element method 
for general elliptic operators hinges   on the loss defined by \eqref{E_hFE_genEll}. \\

\subsection{Contribution and results}
\label{sec:G_convergence}

In the field of machine learning for models characterised by partial differential equations, there is currently significant activity, 
including the development of new methods to solve differential equations, operator learning, and advances in uncertainty quantification and statistical functional inference. Despite the recent advancements in these areas, fundamental mathematical and algorithmic understanding is still evolving.
Several neural network approaches have been introduced over the years, including Deep-Ritz methods,  Physics Informed Neural Networks, Variational PINNs, among others, see e.g.,  \cite{e2017deep}, \cite{Karniadakis:pinn:orig:2019},  \cite{kharazmi2019variational}.  Residual based methods were considered in   \cite{Lagaris_1998}, \cite{Berg_2018},  \cite{Raissi_2018},  \cite{SSpiliopoulos:2018} and their references.  Other neural network methods for differential equations and related problems include, for example, \cite{kevr_1992discrete}, \cite{Xu}, 
\cite{chen2023solving},  \cite{georgoulis2023discrete}, 
\cite{Grohs:space_time:2023}. 
These methods are applied to diverse complex  physical and engineering problems;  for a broader perspective  see e.g., \cite{karniadakis_kevr_2021physics}.   
\subsubsection*{\it Finite Element Training.}
Previously known approaches were based on quadrature-collocation methods  and to Monte-Carlo-Collocation approaches, see the references above and Section 
\ref{training_approaches}. 
\tred{Instead, our} approach retains the simplicity of classical neural network methods for PDEs, uses well established finite element tools (and software) to compute the integrals involved and it gains in efficiency and accuracy. As mentioned, since finite element meshes are required, the applicability of this method 
is limited to low dimensional problems, or to problems where related finite element spaces can be constructed.

\subsubsection*{\it Stability and Convergence.} 
The stability framework introduced in \cite{GGM_pinn_2023} proves to be effective in the current context. We demonstrate that the proposed methods are stable, in the sense made precise in  Proposition \ref{Prop:EquicoercivityofE(2)}. Furthermore, we establish that the resulting approximations converge to the solutions of the PDE, provided that the neural network spaces are chosen to meet specific approximability criteria. The necessary approximation capacity of these neural network spaces aligns with existing results (see Remark \ref{Rmk:NNapproximation}). Our assumptions regarding the PDE solution involve only minimal regularity requirements.

As in \cite{GGM_pinn_2023}, the liminf-limsup framework of De Giorgi—see Section 2.3.4 of \cite{DeGiorgi_sel_papers:2013} and, for example, \cite{braides2002gamma}—used in the 
$\Gamma-$convergence of functionals in nonlinear PDEs and energy minimisation, is particularly valuable. The inclusion of the finite element interpolant in the discrete functionals introduces certain technical challenges, which are addressed in the following section. Importantly, no additional assumptions on the discrete minimisers are necessary to ensure convergence.

We want to highlight that our stability and convergence analysis has practical significance. It assists in determining which energies (or losses) lead to well-behaved (stable) algorithms. This analysis is particularly insightful, as not all seemingly reasonable energies result in stable algorithms, as demonstrated in \cite{GGM_pinn_2023}.


%

\subsubsection*{\it Numerical Performance.} 
In Section \ref{NR}, we present some preliminary numerical results that suggest the proposed method indeed produces accurate and robust algorithms. Integrating finite elements into the deep-Ritz method is particularly simple to implement and offers the expected flexibility in selecting polynomial spaces and finite element quadrature. The performance of the \emph{finite element-deep Ritz method} favourably compares to other training methods in terms of both accuracy and computational execution time.


\subsubsection*{\it Related Literature.}
Combining finite elements and neural networks were considered before mainly in the framework of Variational PINNs \cite {kharazmi2019variational}, 
in the works 
\cite{B_Canuto_P_Vpinn_quadrature_2022} and \cite{Badia_2024}, see also, \cite{HybridNN_Fang2022}, \cite{meethal2023finite}. The interesting approach taken is related to how quadrature rules and  different finite element spaces influence the asymptotic behaviour of  Variational PINNs. In these methods as well the finite element interpolant of the neural network functions is used in the definition of loss. Such methods 
when connected to finite elements introduce a Petrov-Galerkin framework and their stability  relies  on  inf-sup conditions. In addition, detailed numerical results of \cite{B_Canuto_P_Vpinn_quadrature_2022} and \cite{Badia_2024} indicate that upon appropriate parameter tuning such methods are capable to produce very accurate approximations. In \cite{B_Canuto_P_Vpinn_quadrature_2022}  a detailed analysis is presented including error estimates.

Previous works analyzing methods based on  neural network spaces for PDEs include \cite{SSpiliopoulos:2018}, \cite{Mishra_dispersive:2021}, \cite{Karniadakis:pinn:conv:2019}, \cite{shin2023error}, \cite{Mishra:pinn:inv:2022}, \cite{hong2022priori}, and \cite{Mishra_gen_err_pinn:2023}. 
The results in \cite{Mishra_dispersive:2021}, \cite{Mishra:pinn:inv:2022}, and \cite{Mishra_gen_err_pinn:2023} were based on estimates where the bounds depend on the discrete minimisers and their derivatives.  The findings in \cite{hong2022priori}, which involve deterministic training, are related in that they apply to neural network spaces where high-order derivatives are uniformly bounded in suitable norms by design. 
 In  \cite{muller2020deep} $\Gamma$-convergence  was used in the analysis of  deep Ritz methods without training. In the recent work \cite{loulakis2023new}, the $\liminf - \limsup \,$ framework was used in general machine learning algorithms with probabilistic training to derive convergence results for global and local discrete minimisers.  As mentioned, the stability framework and the general plan of convergence based on the the $\liminf - \limsup \,$ framework was first suggested in  \cite{GGM_pinn_2023} where PINN methods were considered  for elliptic and parabolic problems. For recent applications to computational methods where the discrete energies are rather involved,  see \cite{bartels2017bilayer},   \cite{grekas2022approximations}. 
%

\section{Convergence of the discrete minimisers}
\label{sec:G_convergence}

In this section we establish the stability of the algorithm and the convergence of the discrete minimisers to the exact solution of the elliptic problem. 

\subsection{Setting}
Following, \cite{GGM_pinn_2023}, we adopt a key notion of stability motivated by   Equi-Coercivity 
in the $\Gamma-$convergence. This notion   eventually  drives compactness and the convergence of minimisers 
of the approximate functionals. As in  \cite{GGM_pinn_2023} we denote   by $\mathcal{E}_\ell ,$ the approximate functionals where $\ell$ stands for a discretisation parameter. 
 $\mathcal{E}_\ell $ are called stable if the following  two 
key properties hold:  
\begin{enumerate}
	\item [{[S1]}]  If energies  $\mathcal{E}_\ell$ are uniformly bounded
$$\mathcal{E}_\ell [u_\ell] \leq C,
$$  
then  there exists a constant $C_1>0$ and $\ell-$dependent norms (or semi-norms) $V_\ell$ such that 
\begin{align}
&\|u_\ell\|_{V_\ell} \le C_1 . \label{coer:dg_seminorm}
\end{align}
\item [{[S2]}]  Uniformly bounded sequences in $\|u_\ell \| _{V_\ell}$ have convergent subsequences in $H,$
\end{enumerate}
where $H$ is a normed space (typically a Sobolev space) which depends on the form of the discrete energy considered. 
  Additionally, property [S2] implies that even though the norms (or semi-norms)  $\|\cdot \| _{V_\ell}$ vary with $\ell$, they should be designed such that it is possible to extract convergent subsequences in a weaker topology (induced by the space $H$) from uniformly bounded sequences in these norms. \\
  
We shall use standard notation for \tred{Sobolev} spaces
$W^{s, p} (\mathcal {O} ) ,$ having weak derivatives up to order $s$ on $L^p  (\mathcal {O} )$ defined on a set  $\mathcal {O} .$ The corresponding norm is denoted by $\| \cdot \| _ {W^{s, p} (\mathcal {O} )}$ and the seminorm by $| \cdot | _ {W^{s, p} (\mathcal {O} )}.$ The norm of $L^2  (\varOmega )$ will be denoted simply by $\| \cdot \|\, .$\\

Next, assuming we choose the networks appropriately, increasing their complexity should allow us to approximate any $w$ in $H^1$. To achieve this, we select a sequence of spaces $V _{\mathcal{N}}  $ as follows: for each 
  $\ell \in \mathbb N$ we correspond a DNN space  $ V _{\mathcal{N}}  ,$
  which is denoted by  $V_\ell $ with the following property: For each $w\in H_0^1(\varOmega)$ 
  there exists a $w_\ell \in V_\ell$ such that,
    \begin{equation}\label{w_ell_7}\begin{split} 
 \|w_{\ell}-w\|_{H^1(\varOmega)}  \leq \  \beta _\ell \, (w), \qquad 
 \text{and } \ \beta _\ell \, (w) \to 0,  \  \ \ell\to \infty\, .	\end{split}
\end{equation}
If in addition, $w \in W^{m, p} (\varOmega )  $ is in higher order Sobolev space and $1\leq p \leq \infty$
we assume that for $m\geq s+1$  
  \begin{equation}\label{w_ell_7_hSm}\begin{split} 
 \|w_{\ell}-w\|_{W^{s, p} (\varOmega)}  \leq \ \tilde  \beta ^{[m, s,  p]} _\ell \,  | w | _ {W^{m, p} (\varOmega)} , \qquad 
 \text{and } \ \tilde \beta _\ell ^{[m, s,  p]} \, \to 0,  \  \ \ell\to \infty\, .	\end{split} \end{equation}
We do not need specific rates for $ \tilde \beta _\ell ^{[m, s,  p]}\,   ,$
only that the right-hand side of \eqref{w_ell_7_hSm} explicitly depends on the Sobolev norms of $w .$  
This assumption is reasonable given the available approximation results for neural network spaces; see, for example \cite{Xu}, \cite{Dahmen_Grohs_DeVore:specialissueDNN:2022,Schwab_DNN_constr_approx:2022, Schwab_DNN_highD_analystic:2023, Mishra:appr:rough:2022, Grohs_Petersen_Review:2023}, 
and their references. 

\begin{remark}\label{Rmk:NNapproximation}
Despite advances in the approximation theory of neural networks, the current results do not offer sufficient guidance on the specific architectures needed to achieve certain bounds with specific rates. Given that the approximation properties are a significant but separate issue, we have opted to impose minimal assumptions necessary to prove convergence. 
These assumptions can be relaxed by requiring that \eqref{w_ell_7} and \eqref{w_ell_7_hSm} hold specifically for $w = u$, where $u$ is the exact solution of the problem; see Remark \ref{adapt_conv}. 
%
\end{remark}

Furthermore, for each such $\ell$ we associate a finite element space $\S _{h(\ell)}, $ with maximum diameter $h(\ell)$ such that 
$h(\ell)  \to 0,  \  \ \ell\to \infty\, .$ Then
we shall use the compact notation for the minimisation problem
%
\begin{equation}\label{E_ell}
	\min_ {v \in  V _{\ell }} \mathcal{E}_{\ell}(v), \qquad \text{where } \mathcal{E}_{\ell}(v) := \mathcal{E}_{\S _{h(\ell)} }(v)\, .
\end{equation}
The corresponding $\S _{h(\ell)} $-finite element-deep Ritz minimisers are denoted by $u_\ell. $ 
%
\subsection{Stability and Convergence} We start with the stability of the method in the sense made precise below.

\begin{proposition}\label{Prop:EquicoercivityofE(2)} The functional $ \mathcal{E} _\ell $ defined in \eqref{E_ell} is stable with respect to the $ H^1 $-norm, in the following sense: Let $ (v_\ell) $ be a sequence of functions in $ V_\ell $ such that for a constant $ C>0 $ independent of $ \ell $, it holds that
\begin{equation}\label{EquicoercivityofEdelta1(2)}
\mathcal{E}_\ell (v_\ell) \leq C.
\end{equation} 
Then there exists a constant $ C_1>0 $ such that
\begin{equation}\label{EquicoercivityofEdelta2(2)}
\| I_{\S _{h(\ell)} }v_\ell \|_{H^1(\varOmega)} \leq C_1\, .
\end{equation}
\end{proposition}

\begin{proof}
Assuming  $ \mathcal{E}_\ell(v) \leq C $ for some $ C >0 $, then $ \|I_{\S _{h(\ell)} } v||_{H^1(\varOmega)} \leq \tilde{C} $ for some $ \tilde{C}>0 . $ In fact, by the definition of the functional we have   
\begin{equation} \begin{gathered}
\frac 1 2  || \nabla I_{\S _{h(\ell)} } v||_{L^2(\varOmega)}^2 \leq \int_\varOmega I_{\S _{h(\ell)} } ( f v) \dx  +C \leq \| f \| _{L^\infty }  | \varOmega | ^{1/2}\, ||   I_{\S _{h(\ell)} } v||_{L^2(\varOmega)} + C\, . 
\end{gathered}
\end{equation}
The proof is completed by applying the   Poincar\' e  inequality.  
\end{proof}

In the following theorem,  we utilise the  $\liminf$-$\limsup$ framework of $ \Gamma $-convergence, to prove that the sequence $ (I_{\S _{h(\ell)} }  u_\ell) $ where $u_\ell$ are minimisers  of the  functionals $\mathcal{E}_{\ell} $ converges   to the (unique) minimiser of the continuous   functional.

\begin{theorem}[Convergence of the discrete minimisers]\label{Thrm:Gamma_funct} Let $ \mathcal{E},\; \mathcal{E}_{\ell} $ be the energy functionals defined in \eqref{en-func} and 
\eqref{E_ell}
 respectively, and $f\in C^0(\bar \varOmega).$
 Let $ (u_\ell) , $ $u_\ell \in V_{\ell}, $ be a sequence of minimisers of $ \mathcal{E}_{\ell} $ and $$\widehat  u_\ell: =I_{\S _{h(\ell)} } u_\ell \, .$$
Then,  if the finite element spaces are chosen such that $ \underline {h}_{E, \ell} ^ {-1/2}\,  (\tilde  \beta ^{[2,0,  \infty]} _\ell) ^{1-2\epsilon} \leq C,$ where $\underline {h}_{E, \ell} = \min_ {e\in E_{h(\ell)}} h_e,$ we have 
\begin{equation}\label{ConvOfDiscrMinE}
\widehat  u_\ell \rightarrow u ,  \;\; \; \textrm{in} \;\: L^2 (\varOmega), 
\quad 
\widehat   u_\ell \rightharpoonup u  \, ,  \;\; \; \textrm{in} \;\: H^1 (\varOmega),
\qquad \ell \to \infty\, .
\end{equation}
where $u $ is the exact solution of the  problem.
\end{theorem}
\begin{proof}
We show first the  $\liminf$ inequality: We shall show that for all $   v \in H^1 _0 (\varOmega)$ and all sequences 
$( v_\ell )$
 such that $\widehat  v_\ell  \rightarrow    v$ in $L^2(\varOmega), $  where $\widehat  v_\ell: =I_{\S _{h(\ell)} } v_\ell \, , $ it holds that
 \begin{align}\label{liminf_Th}
 \E (v) \leq \liminf\limits_{\ell\rightarrow \infty } \E _\ell ( v_\ell ). 
\end{align}
 We assume there is a subsequence, still denoted by $v_\ell $, 
 such that $ \mathcal{E}_\ell (v_\ell) \leq C $ uniformly in $ \ell $, otherwise $ \mathcal{E}(v) \leq \liminf_{\ell \rightarrow \infty}\mathcal{E}_\ell (v_\ell) = + \infty  .$
 The above stability result, Proposition \ref{Prop:EquicoercivityofE(2)},    implies that $ \|\widehat v_\ell \|_{H^1(\varOmega)} $ are uniformly bounded.  Therefore, up to subsequences,  there exists a $\tilde v\in H^1(\varOmega), $ such that $\widehat v_\ell \rightharpoonup \tilde v $ in $ H^1 $ and $ \widehat v_\ell \rightarrow \tilde v  $ in $ L^2 $, thus $ \widehat v_\ell \rightharpoonup v $ in $ H^1 .$  
 Then we have  
 $ \nabla  \widehat v_\ell    \rightharpoonup \nabla   v$ in $L^2(\varOmega) .$
The term $\int_\varOmega |\nabla \widehat v_\ell  |^2$ is convex which
 implies weak lower semicontinuity \cite{dacorogna2007direct}: 
 $$\liminf\limits_{\ell\rightarrow \infty}
\int_\varOmega |\nabla \widehat v_\ell |^2 \ge \int_\varOmega |\nabla     v|^2.$$
Since  $ \widehat v_\ell  \rightarrow    v$ in $L^2(\varOmega)$  we show below that    
$$\lim \limits_{\ell\rightarrow \infty}
\int_\varOmega  I_{\S _{h(\ell)} } ( v_\ell  f )\dx  = \int_ \varOmega     v f \dx . $$
In fact, we clearly have 
$$\lim \limits_{\ell\rightarrow \infty}
\int_\varOmega  I_{\S _{h(\ell)} } ( v_\ell)\,   f \dx  = \int_ \varOmega     v f \dx , $$
and thus it remains to show
\begin{equation}\label{Th_eq1}
	\lim \limits_{\ell\rightarrow \infty} \Big [
\int_\varOmega  I_{\S _{h(\ell)} } ( v_\ell \,   f ) \dx  - \int_ \varOmega     I_{\S _{h(\ell)} } ( v_\ell)\,   f \dx  \Big ]=0.
\end{equation}  
We shall need some more notation: Let $\{ \Phi _z\} _ {z\in \mathcal {Z}_\ell }, $ be  the Lagrangian basis of $\S _{h(\ell)},$
where $\mathcal {Z}_\ell  $ denotes the set of degrees of freedom. The support of each $\Phi _z$ is denoted by $K_z.$
As is typical,  $K_z$ contains at most a specified number of elements. Thus 
$$ \sup _{ z \in \mathcal {Z}_\ell } \, | K_z|  \to  0,  \qquad \ell \to \infty\, . $$ Using this notation, we observe
\begin{equation*}
	\begin{split}
		\Big  |
\int_\varOmega & I_{\S _{h(\ell)} } ( v_\ell \,   f ) \dx  - \int_ \varOmega     I_{\S _{h(\ell)} } ( v_\ell)\,   f \dx  \Big |\\
& \leq 	\Big  |
\int_\varOmega \sum _{ z\in  \mathcal {Z}_\ell }   v_\ell (z) \,   f (z) \, \Phi _z (x)   \dx  - \int_\varOmega  \sum _{ z\in  \mathcal {Z}_\ell }   v_\ell (z) \,   f (x) \, \Phi _z (x)   \dx    \Big |\\
& \leq 	\sum _{ z\in  \mathcal {Z}_\ell }  
\int_\varOmega  \big |  f (z) - f(x) \big | \,  \Big  |   v_\ell (z) \,   \Phi _z (x)   \Big |     \, \dx \\
	& \leq 	\sum _{ z\in  \mathcal {Z}_\ell }  
\sup _ {x\in K_z}  \big |  f (z) - f(x) \big | \, \int_{K_z}  \Big  |   v_\ell (z) \,   \Phi _z (x)   \Big |     \, \dx \\
& \leq  \sup _{ z \in \mathcal {Z}_\ell }  \sup _ {x\in K_z}  \big |  f (z) - f(x) \big | \,  	\sum _{ z\in  \mathcal {Z}_\ell }  
 \int_{K_z}  \Big  |   v_\ell (z) \,   \Phi _z (x)   \Big |     \, \dx \, .\\
	\end{split}
\end{equation*}
We introduce the following notation: $\overline h_z = \max _{ K\subset K_z} h_K ,$  $\underline  h_z = \min _{ K\subset K_z} h_K . $
By our assumptions on the finite element spaces there holds for a $\beta >0,$ constant independent  of $h$ (and thus of $\ell$) that 
$$ \overline h_z \leq \beta \underline  h_z. $$
We have now using standard homogeneity arguments, see \cite[Section 4.5]{BrennerSFE}, 
\begin{equation*}
	\begin{split}
   	\sum _{ z\in  \mathcal {Z}_\ell }   
 \int_{K_z}  \Big  |   v_\ell (z) \,&   \Phi _z (x)   \Big |      \, \dx \,
		  \leq C 	\sum _{ z\in  \mathcal {Z}_\ell }   
  \overline h_z   ^ {\ d }   \big  |   v_\ell (z) \,     \big |  \\
   &\leq C 	\Big ( \sum _{ z\in  \mathcal {Z}_\ell }   
  \overline h_z   ^ {\ d }   \big  |   v_\ell (z) \,     \big |   ^2   \Big ) ^{1/2} \,
  \Big ( \sum _{ z\in  \mathcal {Z}_\ell }   
  \overline h_z   ^ {\ d }      \Big ) ^{1/2}\\
   &\leq C 	\Big ( \sum _{ z\in  \mathcal {Z}_\ell }   
  \overline h_z   ^ {\ d }   \big  \|   I_{\S _{h(\ell)} }  v_\ell   \,     \big \|_ {L^\infty ( K_z)}  ^2   \Big ) ^{1/2} \,
  \big | \varOmega        \big | ^{1/2}\\
  &\leq C 	\Big ( \sum _{ z\in  \mathcal {Z}_\ell }   
  \overline h_z   ^ {\ d }  \underline h_z   ^ {\ - d }  \big  \|   I_{\S _{h(\ell)} }  v_\ell   \,     \big \|_ {L^2 ( K_z)}  ^2   \Big ) ^{1/2} \,
  \big | \varOmega        \big | ^{1/2}\\
  &\leq C \beta ^ d 	\Big ( \sum _{ z\in  \mathcal {Z}_\ell }   
   \big  \|   I_{\S _{h(\ell)} }  v_\ell   \,     \big \|_ {L^2 ( K_z)}  ^2   \Big ) ^{1/2} \,
  \big | \varOmega        \big | ^{1/2} \leq C \, \big  \|   I_{\S _{h(\ell)} }  v_\ell   \,     \big \|_ {L^2 ( \varOmega )}  \big | \varOmega        \big | ^{1/2}\, . 
	\end{split}
\end{equation*}
We conclude therefore that 
\begin{equation*}
	\begin{split}
		\Big  |
\int_\varOmega & I_{\S _{h(\ell)} } ( v_\ell \,   f ) \dx  - \int_ \varOmega     I_{\S _{h(\ell)} } ( v_\ell)\,   f \dx  \Big |\\
& \leq  C\, \sup _{ z \in \mathcal {Z}_\ell }  \sup _ {x\in K_z}  \big |  f (z) - f(x) \big | \,  	\big  \|   I_{\S _{h(\ell)} }  v_\ell   \,     \big \|_ {L^2 ( \varOmega )} \, . 
	\end{split}
\end{equation*}
Since $\big  \|   I_{\S _{h(\ell)} }  v_\ell   \,     \big \|_ {L^2 ( \varOmega )} $ is bounded and $f$ is uniformly continuous, \eqref{Th_eq1} follows and thus \eqref{liminf_Th} is valid. 

Let $ w \in H^1_0 (\varOmega)  $ be arbitrary;  we will show the existence of a recovery sequence  $(w_\ell)$, such that  $ \mathcal{E}(w) = \lim_{\ell \rightarrow \infty} \mathcal{E}_\ell (w_\ell)  .$  
 For each $\delta >0$ we can select a smooth enough mollifier $w_\delta  \in   C^m_0 (\varOmega),$ $m>2, $ such that   
 \begin{equation}
\begin{split}
	\|&w -w_\delta\|_{H^1( \varOmega ) }  \lesssim \delta \, , \quad \text{and,}\\
 |& w_\delta   |_{  H^s(\varOmega)} \lesssim \frac{1}{\delta^{s-1}} |w|_{  H^1(\varOmega)} .
	\end{split}
\end{equation}
For $w_\delta, \tred{recalling} $ \eqref{w_ell_7_hSm}, there exists $ w_{\ell, \delta} \in V_\ell $ such that
$$\|w_{\ell, \delta }-w_\delta \|_{H^1(\varOmega)}  \leq \ \tilde  \beta _\ell \, \| w_\delta\| _{  H^s(\varOmega)}
\leq \ \tilde  \beta _\ell \frac 1 {\delta ^{s-1}}\, \| w \| _{  H^ 1(\varOmega)},\qquad 
 \text{and } \ \tilde \beta _\ell \, (w) \to 0,  \  \ \ell\to \infty\, .$$
 \tred{Next, we distinguish two cases regarding the regularity of the discrete neural network spaces: (i) $ V_{\ell} \subset H^2(\varOmega)$
 and (ii) elements of $ V_{\ell} $ which are only Lipschitz continuous. The second case corresponds to the choice of ReLU activation function. }\\

 \noindent
 \tred{\emph{Case 1:  $ V_{\ell} \subset H^2(\varOmega).$} } \\
  Upon noticing that $w_\delta $ has zero boundary trace, but $ w_{\ell, \delta }$ 
 has not,  we first observe that 
 \begin{equation} \label {smoothNN}
 	 \|\nabla \big ( w_{\ell, \delta }-I_{\tilde \S _{h(\ell)} }  w_{\ell, \delta } \big )\|
 \leq \ C  h(\ell)  \, | w_{\ell, \delta } | _{  H^2(\varOmega)} 
\leq \ \ C  h(\ell)  (1 + \tilde  \beta _\ell )\frac 1 {\delta  } \, \| w \| _{  H^ 1(\varOmega)},\qquad 
\end{equation}
 where $ \tilde \beta _\ell \, (w) \to 0,  \  \ \ell\to \infty\, .$ \tred{ Also, using the fact that  $ I_{\tilde \S _{h(\ell)} } w_{\ell, \delta } - I_{  \S _{h(\ell)}}  w_{\ell, \delta } $ is zero at all nodal points except those at the boundary, we first observe that for any element $K$ which has a face $e$  on $\partial \Omega$ we have,
 \begin{equation*}	
 \begin{split}
 	 \|\nabla \big ( I_{\tilde \S _{h(\ell)} } w_{\ell, \delta } -& I_{  \S _{h(\ell)}}  w_{\ell, \delta } \big )\|^2_{L^2(K)}  
 \leq \ C  \underline {h}_{K} ^{-2} \, |K|  | w_{\ell, \delta } | ^2_{  L^\infty (e)} \leq \ C  \underline {h}_{K} ^{-1} \, |e|  | w_{\ell, \delta } | ^2_{  L^\infty (e)}
\, ,\qquad 
 \end{split} 
\end{equation*} 
and therefore,
 \begin{equation*}	
 \begin{split}
 	 \|\nabla \big ( I_{\tilde \S _{h(\ell)} } w_{\ell, \delta } -& I_{  \S _{h(\ell)}}  w_{\ell, \delta } \big )\|
 \leq \ C |\partial \Omega| ^{1/2} \underline {h}_{E, \ell} ^ {-1/2} \, | w_{\ell, \delta } | _{  L^\infty (\partial \varOmega)} \\
& =\ C  \underline {h}_{E, \ell} ^ {-1/2} \, | w_{\ell, \delta } -w_\delta | _{  L^\infty (\partial \varOmega)} 
\leq \ \ C  \underline {h}_{E, \ell} ^ {-1/2}\,  \tilde  \beta ^{[2,0, \infty]} _\ell \frac 1 {\delta^{d +1}  } \, \| w \| _{  H^ 1(\varOmega)}\, .\qquad 
 \end{split} 
\end{equation*} }
 Choosing $\delta $ appropriately, e.g., $\delta = \max \{ \tilde \beta _\ell ^{1/2}, h(\ell) ^{1/2}, \big (   \underline {h}_{E, \ell} ^ {-1/2}\,  (\tilde  \beta ^{[2, 0, \infty]} _\ell) ^{1-\epsilon} \big ) ^ {1/d+1} \ \},$  as function of $\tilde \beta _\ell$  and $ h(\ell) $ we can ensure that $w_\ell =w_{\ell, \delta}$ 
satisfies, 
\begin{equation}\label{prooflimsupeq2}
 \mathcal {E} _\ell (w_\ell )  \rightarrow \mathcal {E} (w ) \, , \  \ \ell\to \infty\, .
\end{equation}\\

 \noindent
 \tred{\emph{Case 2:  ReLU activation function.}}\\
 \tred{If  $\sigma (x) =ReLU(x) = \max \{ 0, x\}$ then elements of $ V_{\ell} $ are only Lipschitz continuous. }
 In this case, one needs to modify the argument based on the first inequality of 
\eqref{smoothNN}. 
To this end, we introduce first  a Cl\'ement type interplant, see e.g.,  \cite{Clement_1975}, \cite{Carstensen_Clement_1999},\cite{Verfurth_Clem_Inter_1999}, by 
\begin{equation}
I_{ C _{h(\ell)}} g =	\sum _{ z\in \tilde  {\mathcal {Z}} _\ell }   \overline     g (z) \, \Phi _z (x)   \dx, \qquad \overline      g (z) = \frac 1 {|K_z|} \int _{K_z} g(y) \d y\, ,
\end{equation}
where as before, $K_z$ denotes the support of each $\Phi _z$ and $\tilde   {\mathcal {Z}}_\ell $ denotes the set of degrees of freedom of $\tilde \S _{h(\ell)} \, .$ It is well known that the regularity required by this  interpolation operator 
 is just $g \in L^ 1 (\varOmega)$  and that it satisfies stability and error estimates under minimal conditions.   

Employing  $I_{ C _{h(\ell)}} $ we  modify the argument based on the first inequality of 
\eqref{smoothNN} as follows 
 \begin{equation} \label {non_smoothNN_1}
 \begin{split}
 	 \|\nabla \big ( w_{\ell, \delta }-I_{\tilde \S _{h(\ell)} }  w_{\ell, \delta } \big )\| 
 	& \leq  \|\nabla \big ( w_{\delta }-I_{\tilde \S _{h(\ell)} }  w_{ \delta } \big )\| \\
 	&\quad  +  \|\nabla \big ( [ w_{\ell, \delta } - w_{ \delta } ] -I_{\tilde \S _{h(\ell)} }   [ w_{\ell, \delta } - w_{ \delta } ] \big )\|\\
 	& \leq  \|\nabla \big ( w_{\delta }-I_{\tilde \S _{h(\ell)} }  w_{ \delta } \big )\| \\
 	&\quad  +  \|\nabla \big (I_{ C _{h(\ell)}} [ w_{\ell, \delta } - w_{ \delta } ] -I_{\tilde \S _{h(\ell)} }   [ w_{\ell, \delta } - w_{ \delta } ] \big )\| \\
 	&\quad  +  \|\nabla \big ( I_{ C _{h(\ell)}} [ w_{\ell, \delta } - w_{ \delta }    ] \big )\| \, .
 		 \end{split}
%
%
%
%
\end{equation}
For the first term, we have, 
 \begin{equation} \label {non_smoothNN_2}
 	 \|\nabla \big ( w_{  \delta }-I_{\tilde \S _{h(\ell)} }  w_{ \delta } \big )\|
 \leq \ C  h(\ell)  \, | w_{  \delta } | _{  H^2(\varOmega)} 
\leq \ \ C    \frac { h(\ell) } {\delta  } \, \| w \| _{  H^ 1(\varOmega)}\, .\qquad 
\end{equation}
The stability of Cl\'ement  interplant implies, 
 \begin{equation} \label {non_smoothNN_3}
 	\|\nabla \big ( I_{ C _{h(\ell)}} [ w_{\ell, \delta } - w_{ \delta }    ] \big )\| 
 \leq \ C  \|w_{\ell, \delta }-w_\delta \|_{H^1(\varOmega)}
 \leq C \, \tilde  \beta _\ell \, \| w_\delta\| _{  H^s(\varOmega)}
\leq C \, \tilde  \beta _\ell \frac 1 {\delta ^{s-1}}\, \| w \| _{  H^ 1(\varOmega)}\, .
\end{equation}
It remains to estimate the second term on the right hand side of \eqref{non_smoothNN_1}. To this end let $K$ a fixed element of the triangulation. We then have \tred{for $g= w_{\ell, \delta } - w_{ \delta },$}
\begin{equation*}
	\begin{split}
	\|\nabla \big (I_{ C _{h(\ell)}} \, g - &I_{\tilde \S _{h(\ell)} }   \, g \big )\| _{L^2(K)}
 =  	\Big   \| 
 \sum _{ z\in  \tilde {\mathcal {Z}} _\ell }  [ \overline g(z)  - g(z)  ] \nabla\Phi _z \Big  \| _{L^2(K)}   \\
& \leq 	\sum _{ z\in  \tilde {\mathcal {Z}} _\ell }  
  \big |   \overline g(z)  - g(z)  \big | \,  \big  \| \nabla\Phi _z \big  \| _{L^2(K)}  \\
  &  \leq  	\sum _{ z\in  \tilde {\mathcal {Z}} _\ell }  
\frac 1 {|K_z|} \int _{K_z}   \big |     g(y)  - g(z)  \big | \, \d y \,  \big  \| \nabla\Phi _z \big  \| _{L^2(K)}  \\
&  \leq  C	\sum _{ z\in  \tilde {\mathcal {Z}} _\ell }  
\frac 1 {|K_z|} \int _{K_z}  \,  \overline h_z \big |   \nabla   g  \big |_{\tred{L^\infty(K_z)}} \, \d y \, h_K ^{-1}\,  \big  \| \Phi _z \big  \| _{L^2(K)}  \\
&  \leq  C	  \, \max _{z :\ K \subset K_z }\, \overline h_z \big |   \nabla   g  \big |_{\infty, K_z} \, \, \underline  h_z ^{-1}\,  \big  | K \big  | ^{1/2}  \leq  C	 \beta \,    \big |   \nabla   
g  \big |_{\tred{L^\infty(\varOmega)}} \, \,   \big  | K \big  | ^{1/2}\, . \\
	\end{split}
\end{equation*}
Where we have used the fact $ \overline h_z \leq \beta \underline  h_z $
and that given the family of triangulations,  for each  $K$ the number of $K_z$ such that $K \subset K_z$ is finite and fixed.
We conclude therefore that 
 \begin{equation} \label {non_smoothNN_4}
 \begin{split}
 	\|\nabla \big (I_{ C _{h(\ell)}} [ w_{\ell, \delta } - w_{ \delta } ] -I_{\tilde \S _{h(\ell)} }   [ w_{\ell, \delta } - w_{ \delta } ] \big )\| 	& \leq C\, 
	\big |w_{\ell, \delta } - w_{ \delta } \big |_{\tred{W^{1, \infty}(\varOmega)}} \,  \big  | \varOmega  \big  | ^{1/2}\, .
 	 		 \end{split}
\end{equation}
This bound completes the proof, as for ReLU networks it is expected that 
\eqref{w_ell_7_hSm} will hold for $W^{s, p}= W^{1, \infty}, $
see e.g. \cite{guhring2020error} and its references. \\

To conclude the proof, let $u\in H^1_0 (\varOmega)  $ be the unique solution of \eqref{mm} and consider  the sequence of the discrete minimisers $(u_\ell)\, .$ Then, 
$$ \mathcal{E}_{ \ell} (u_\ell) \leq \mathcal{E}_{ \ell} (v_\ell), \qquad \text{for all } v_\ell \in V_\ell\, . $$
Specifically,  
 $ \mathcal{E}_{ \ell}(u_\ell) \leq \mathcal{E}_{ \ell} (\tilde u_\ell),  $ 
where $\tilde u_\ell $ is the recovery sequence constructed above corresponding to $w=u.$  
Since $ \mathcal {E} _\ell (\tilde u_\ell)  \rightarrow \mathcal {E} (u), $ the sequence $( \mathcal {E} _\ell (\tilde u_\ell)  )$ is bounded 
and therefore,   the discrete energies are uniformly  bounded. Then the stability result  
Proposition \ref{Prop:EquicoercivityofE(2)}, implies that
\begin{align}
 \norm{ I_{\S _{h(\ell)} }  u_\ell }_{H^1(\varOmega)}   < C,  
\end{align} 
uniformly. We apply  the  Rellich-Kondrachov theorem, \cite{Evans},  and similar arguments as in the proof of  $\liminf$ inequality to conclude the following:  There exists $ \tilde   u\in H^1 (\varOmega ) $ 
such that $ I_{\S _{h(\ell)}} u_\ell  \rightarrow   \tilde   u$ in 
$L^2(\varOmega) $ up to a subsequence not relabeled here. Furthermore, by the trace inequality, the fact that $ I_{\S _{h(\ell)}} u_\ell$ have zero trace and have uniformly bounded $H^1$ norms we conclude that $\tilde u $ has zero trace. 
Next we show that   $ \tilde   u = u$ where $u $ is the global minimiser of $ \E. $ 

Let $w \in H^1_0 (\varOmega) $, and   $ w_\ell \in V_\ell $ be its recovery sequence constructed above.
Therefore,    the $\liminf$ inequality 
and the fact that $ u_\ell $ are   minimisers of the $\E _{\ell},$ imply that
\begin{align}
 \E (  \tilde   u ) \le  \liminf_{\ell \rightarrow \infty }  \E _{ \ell} ( u_\ell ) 
  \le  \limsup_{\ell \rightarrow \infty }  \E _{ \ell} ( u_\ell ) 
\le  \limsup_{\ell \rightarrow \infty }  \E _{ \ell} ( w_\ell ) 
= \E  (  w ). 
\end{align}
Since  $ w \in H^1_0 (\varOmega) $ is arbitrary,  $ \tilde  u$ is a  minimiser of $ \E ,$ and since $u $ is the unique global minimiser of $ \E  $ on  $H^1_0 (\varOmega) $ we have that $\tilde  u=u $. Since all subsequences have the same limit, the entire sequence converges $ I_{\S _{h(\ell)} } u_\ell \to u.$

\end{proof}

 \begin{remark}\label {adapt_conv}
 The final argument of the proof can be carried out by using only a recovery sequence for $w=u,$ $u$ being the exact solution. If $\{ \tilde u_\ell\}$ is such a  sequence, we will have 
\begin{align*}
 \E (  \tilde   u ) \le  \liminf_{\ell \rightarrow \infty }  \E _{ \ell} ( u_\ell ) 
  \le  \limsup_{\ell \rightarrow \infty }  \E _{ \ell} ( u_\ell ) 
\le  \limsup_{\ell \rightarrow \infty }  \E _{ \ell} (\tilde u_\ell ) 
= \E  (  u ). 
\end{align*}
Since    $u $ is the unique global minimiser of $ \E  $  we have that $\tilde  u=u $. \\
  \end{remark}	
 
 \begin{remark}[Convergence for General Elliptic Problems]\label {gen_ell_conv}
 The convergence results generalise in a straightforward manner in the case of general elliptic operators with 
 constant coefficients, \eqref{ellipticPDE}. In the case of variable $ a_{ij},$ $c,$ one has to employ appropriate quadrature rules
 as discussed in Section \ref{gen_ell}. The convergence proof is based on similar arguments as above, however, 
 a series of technical estimates and specific approximation properties of the quadrature need to be used. 
 In fact, one has to establish for both liminf and limsup inequalities  convergence of the form
 \begin{equation*}  \begin{split}
 &	\lim _{\ell \to \infty }  | B (I_{\S _{h(\ell)}} (g_\ell) I_{\S _{h(\ell)}} ( g_\ell) ) - B_{h(\ell)} (I_{\S _{h(\ell)}}( g_\ell),I_{\S _{h(\ell)}}( g_\ell)  ) | =0
 \, .
 \end{split}
\end{equation*}
 To this end, similar estimates as in  \cite[Theorem 29.1] {CiarletFE} along with homogeneity arguments need to be employed. 
 \end{remark}

\section{Numerical Results}\label{NR}
In the sequel we compare the aforementioned training methods, Section~\ref{training_approaches},  to approximate  the minimiser 
of equation~\eqref{en-func}. 
We have chosen the right hand side 
of \eqref{ivp-AC} such that the energy  minimizer has the form
\begin{align}\label{eq:uexact}
u_e(x_1, x_2) = \sin(2\pi x_1)\sin(2\pi x_2), \text{ } (x_1, x_2) \in \varOmega = [0,1]^2. 
\end{align}
We seek to optimise the parameters of the Neural Network $u_\theta$
employing Monte-Carlo, quadrature  and finite elements training.
We compare their accuracy, their computational cost and their behaviour 
in the optimisation process of the network parameters. After testing various optimisers and learning rate schedulers we have chosen the \textbf{Adam} optimizer with cyclic 
learning rate policy (CLR), ranging the learning rate between $10^{-3}$ and $10^{-5}$, as it has provided a more efficient  training. 
The Neural Network is developed through the framework provided by PyTorch,
\cite{paszke2019pytorch}, with parameters of float32 precision executing the code in cpu.
The applied quadrature rules and the finite element spaces are based  on triangulations of the domain as in Fig.~\ref{fig:mesh}, 
where the unit square is divided in $2 M^2$ triangles. In the sequel 
we will perform simulation with $M=$ 20, 40, 60, 80, 100, 120. 
For details for the applied quadrature rules see, e.g., \cite{StrangFix,quadraturesSite}.


\begin{figure}
\centering
  \includegraphics[width=0.35\linewidth]{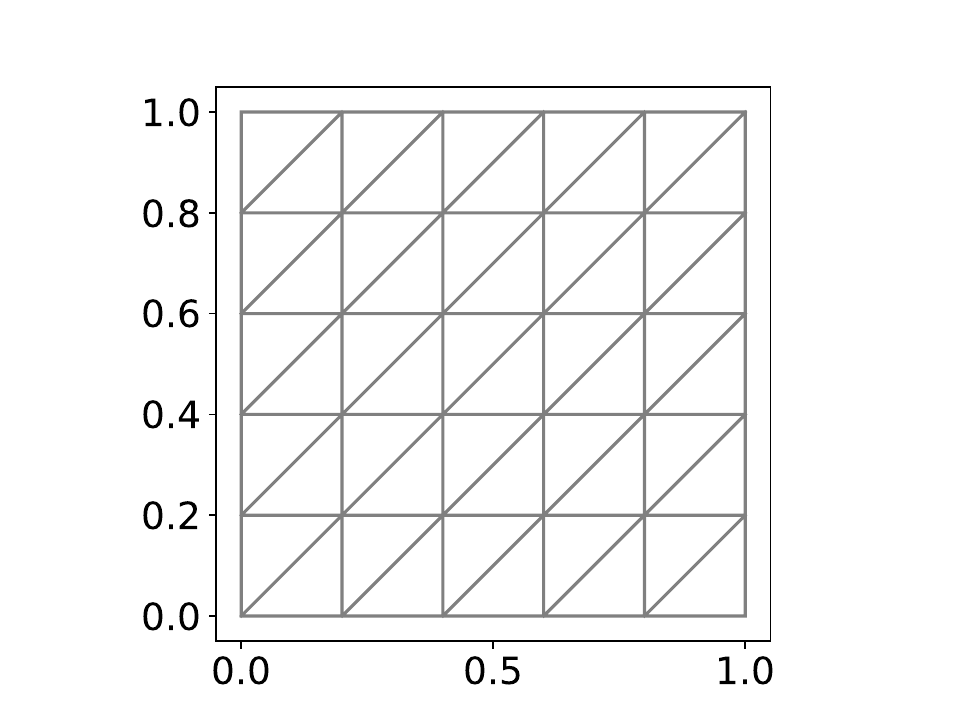}
  \caption{Uniform mesh: The unit square is divided in $2 M^2$ triangles (cells), here $M=5$. In the simulation that follow 
  $M=$ 20, 40, 60, 80, 100, 120.} 
  \label{fig:mesh}
\end{figure}

\begin{figure}
\centering
  \includegraphics[width=0.8\linewidth]{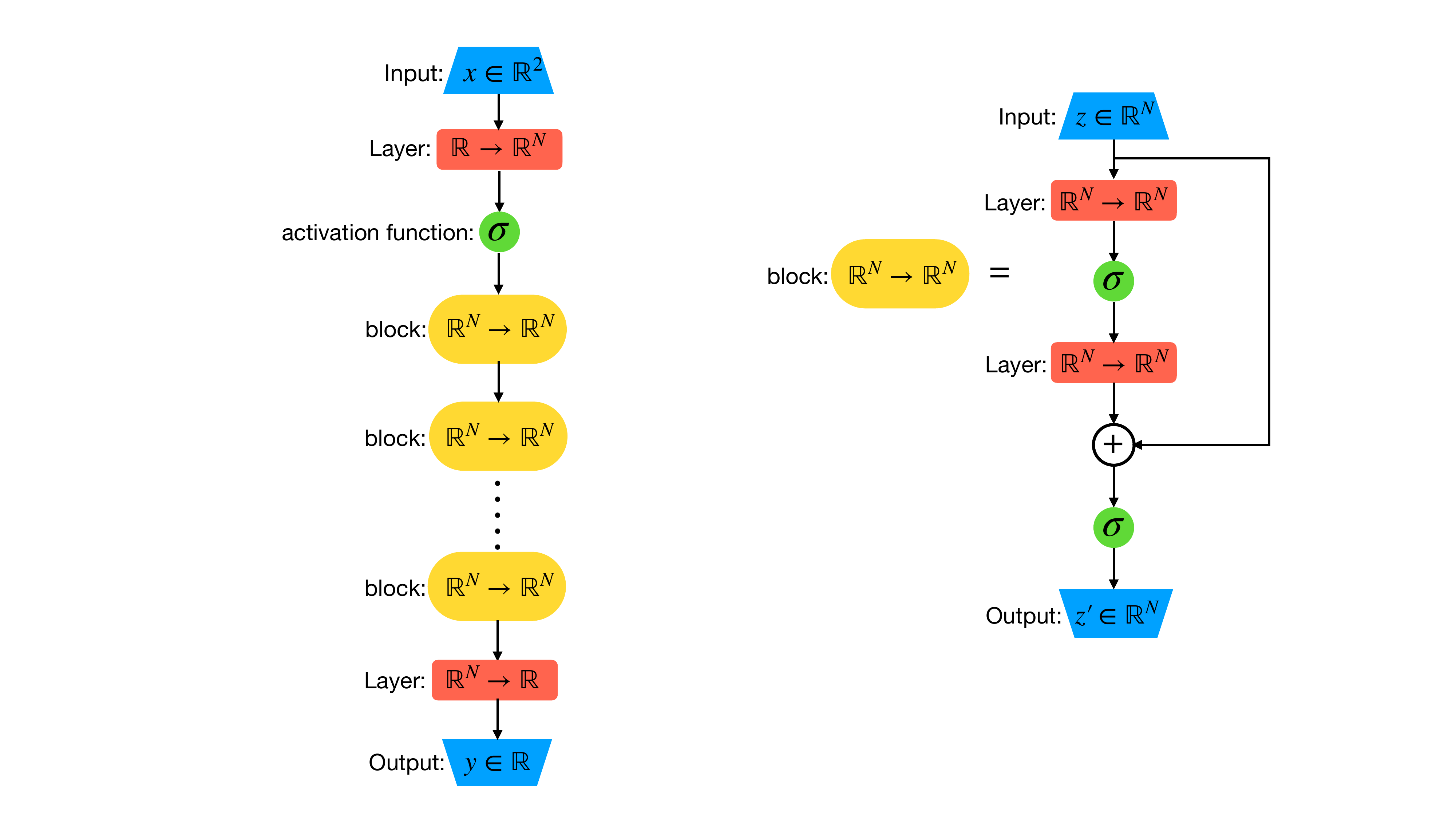}
  \caption{
  A Residual Neural Network.
  } 
  \label{fig:ResNet}
\end{figure}

For the Neural Network architecture we have employed a Residual Neural Network, see Fig.~\ref{fig:ResNet}, with $m$ number of blocks.  The $k$-th block is defined as 
 \begin{align}
 \text{bl}_k(z) =\sigma \left( \left( W_{2k} \sigma \left( W_{1k} z + b_{1k} \right) + b_{2k} \right) + z \right), 
 \end{align}
 where $z, b_{jk} \in \R^N$  and $W_{jk} \in \R^{N \times N}$, $j=1,2$, and $\sigma$ is the activation function $\tanh(.)$.
 The overall architecture can be described by the map $\mathcal{C}_L : \R ^2 \to \R$, specifically 
 \begin{equation}
 \mathcal{C}_L:= C_o    \circ \text{bl}_m \cdots  \circ \text{bl}_1 \circ\sigma \circ C_i,
\end{equation}
where $C_i$, $C_o$ are the input, output layers respectively with $C_i : \R^2 \rightarrow \R^N$ and $C_o: \R^N \rightarrow \R$. 
 Varying the number of blocks from $1$ to $4$ and fixing $N=64$ we optimize the network parameters for each 
training  method. \tred{This architecture represents a slight modification of the generic design outlined in Section 1.2. Our computational results showed that spaces based on Residual networks performed similarly and, in some cases, exhibited better behavior. As mentioned in Section 1.2, our convergence results are not dependent on the particular neural network architecture selected for the discrete spaces.}
%
%
%

\subsection*{Training through Monte-Carlo/Collocation}
This is the most straightforward  and widely used approach.
In every iteration $2M^2$ and $4M$   random points\footnote{The formula of $2M^2 + 4M$  total points is adopted for comparison with the next methods.}
 are generated for the interior and the boundary of $\varOmega$
respectively. From these $2M^2$ points the loss function is computed as in equation~\eqref{prob_E} and  the $4 M$ points impose the 
boundary conditions weakly by adding to the loss function the following term
\begin{equation}\label{bcs_penalty}
 \frac c N  \sum _{i=1 } ^N     | u_\theta(x_i )| ^2 
\end{equation}
where $N=4 M$, $x_i\in \partial \varOmega$ are the corresponding  random points at the boundary and 
$c$ is the penalty parameter imposing weakly the Dirichlet boundary conditions. In our simulations $c=40$ as it provides better 
approximations results.
In Figure~\ref{fig_errors_random} the 
$L^2$ error between $u_\theta$ and $u_e$ of equation~\ref{eq:uexact}, is illustrated for $M=20, 40, 60, 80, 100, 120$
varying the blocks number of the Residual Neural Network from $1$ to $4$. 
We start the training procedure with $M=20$  optimizing the parameters for 40000 epochs.
The optimized parameters are the input for the next step with $M=40$, re-optimizing  the network parameters for $20000$ epochs. 
This initialization from the previous (smaller) number of collocation points is repeated as $M$ increases and the parameters are retrained for 20000 epochs. 
The $L^2$-error  is of the order of $10^{-2}$ and for $M\ge 40$ does not reduce significantly
 for the four compared residual networks, Fig.~\ref{fig_errors_random}. The values of the loss functions approach each other
for $M \ge 80$ while the execution time increases almost linearly with respect to the blocks number. 

\begin{figure}
\centering
  \includegraphics[width=1\linewidth]{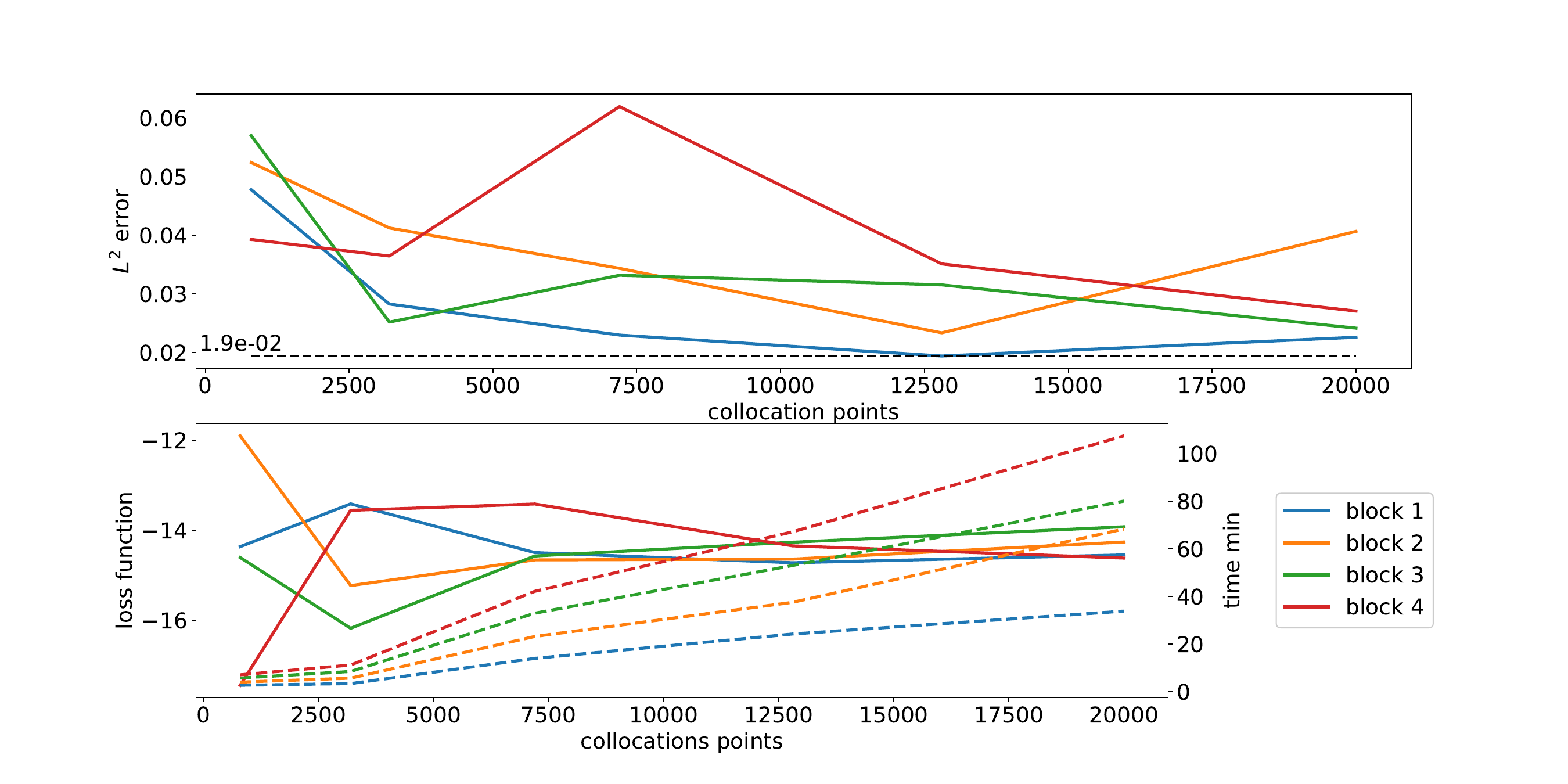}
  \caption 
{\emph{Monte-Carlo Collocation.\/} Energy minimization through collocation points varying collocation points and  the blocks number of the Residual Network architecture.
   Top image:  For a given number of collocation points   $\norm{u_\theta - u_e}_{L^2(\varOmega)}$ is computed
 where $u_\theta, u_e$  denote the discrete, exact minimizers respectively. The minimum error $1.9 \cdot 10^{-2}$ is achieved 
for  $1$ block and $M=80$, i.e. 13120 collocation points.
 Bottom image: Solid curves illustrate loss function values after the final epoch
  iteration (left vertical axis) and dashed curves the total execution time (right vertical axis).
  } 
  \label{fig_errors_random}
\end{figure}

\subsection*{Training through quadrature/collocation}
Next we examine the training through quadrature of first degree, see Section 1.4.1. The number of integration rule  points  are $2M^2$ for the interior of $\varOmega$ and $4M$
for the boundary,  as in the Monte-Carlo method, 
but rearranged in way to integrate polynomial of first degree exactly.
Given that in this case there is an underlying mesh, the boundary conditions can be  imposed weakly
by including the term of eq.~\eqref{Nitsche_bcs} (Nitsche's method) with $\alpha=40$. Then, the $L^2$-error is of the  order of $\sim 10^{-3}$, Fig.~\ref{fig_errors_autograd}. Here 
without increasing the computational cost a better approximation is achieved. On the other hand, similarly to the previous method there is no significant gain by enriching
 the Residual Network with extra blocks while the $L^2-$error is non monotonic as more integration points are added.
 We noticed that lower accuracy is achieved when the boundary conditions are imposed from eq.~\eqref{bcs_penalty}  
where the error is accumulated at the boundary.

\subsection*{Training through Finite Elements}
In the third examined method a finite element space of piece-wise linear polynomials is employed. The boundary conditions are imposed weakly through \eqref{E_hFE_nitsche}, see the discussion in Section 1.5. 
Here the $L^2-$error decreases monotonically as more integration points are added, and lies between  $10^{-3}$ and $10^{-4}$,
reaching the minimum value of $6.8 \cdot 10^{-4}$ for $4$ blocks of the Residual Network, see red curve of Fig.~\ref{fig_errors_interpolant}.
Here as more blocks are added a better approximation is achieved  but is not improved for more than 4 blocks.
 It is remarkable the significantly
decreased computational time compared to the two previous methods.  The computational time is approximately 
4 times shorter.
This is a result of computing $\nabla I_h u_\theta$ instead of 
$\nabla u_\theta$, where the former is a straightforward computation. 
Note that in this case, the integrand is computed exactly and from standard error estimates  an error of order $O(h)$, $h$ being the mesh size, 
\tred{is resulting from} the interpolation of  $u_\theta$. 

\subsection*{Comparison of the methods and further computations}
In Fig.~\ref{fig:ue_and_diff} it is illustrated the graphs of $u_e$ and of the pointwise error $|u_\theta - u_e|$ squared,  for all the compared training methods.
For each method we have chosen the  architecture with the minimum $L^2$ error, see Figs~\ref{fig_errors_random}, \ref{fig_errors_autograd} 
and \ref{fig_errors_interpolant}. In this figure it is apparent that training with quadrature collocation  
improves the squared error of Monte Carlo collocation,  reducing it up to two orders of magnitude. For the training through finite elements
the error is reduced up to 3 orders of magnitude compared to Monte Carlo collocation. 
It seems that, within the framework of 32 bit computations we are testing, the suggested deep Ritz finite element method  provides quite accurate approximation of $u_e.$ 
Furthermore, training through finite elements results in a significant reduced computational cost, compare execution times from bottom pictures of 
 Figs~\ref{fig_errors_random}, \ref{fig_errors_autograd} 
and \ref{fig_errors_interpolant}.\\

It is natural to ask how one can improve the precision of the approximation when training through quadrature collocation is used by increasing the accuracy of the quadrature rule employed. 
 For that purpose, we examine  
the training procedure through quadrature  under numerical integration ranging from order $2$ to order $5$ fixing the blocks number to $4$, Fig.~\ref{fig:autograd_over_id}. 
We notice that 
a slightly better approximation of $u_e$ is achieved with, a minimum $L^2$ error of $8.9\cdot 10^{-4}$, compared  to the integration 
rule of order $1$, Fig.~\ref{fig_errors_autograd}. 
However, this error is still higher than the minimum $L^2-$error  for the training through finite elements,  i.e. $7.2 \cdot 10^{-4}$
in Fig.~\ref{fig_errors_interpolant}. \\

To this end, we perform computational experiments with  finite element spaces with  polynomial degree of order $2$. 
Therefore, we employ an integration rule of degree 2, Fig~\ref{fig:errors_interpolant2}. Here the $L^2-$ error is reduced even more 
from $7.4 \cdot10^{-4}$ to $5.8 \cdot10^{-4}$, compare Figs.~\ref{fig_errors_interpolant} and \ref{fig:errors_interpolant2}, 
showing a monotonic error decrease as more integration points are added, capturing a better approximations as more blocks
are added.\\

We want to stress that these results are preliminary and intended to demonstrate the potential of the proposed method. A detailed computational analysis of the method's behaviour and its comparison with other neural network based approaches is beyond the scope of this paper, as it will naturally vary depending on the specific nature of the PDE being approximated and the choice of the approximating method.

\begin{figure}
\centering
 \includegraphics[width=1\linewidth]{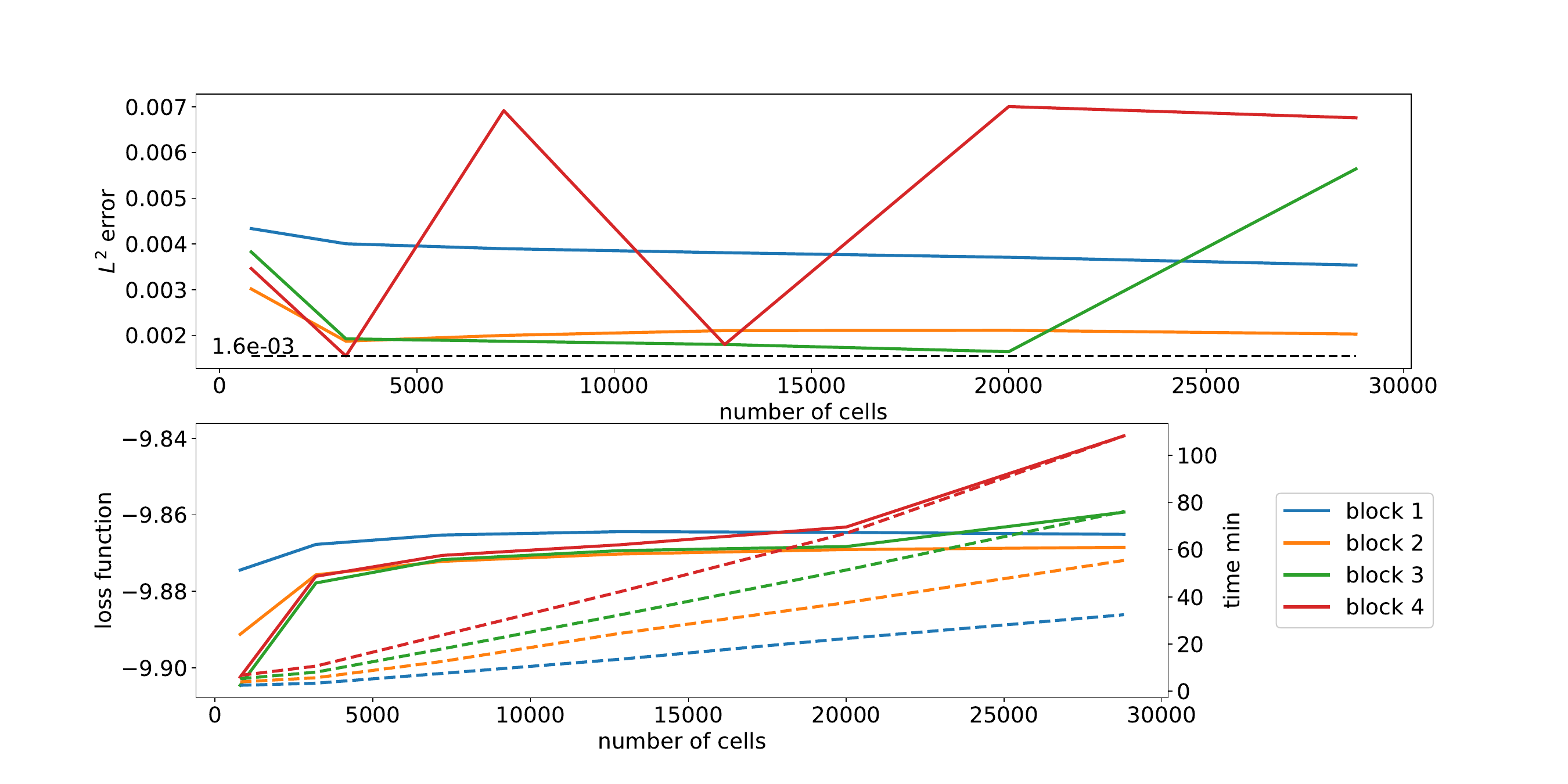}
    \caption{\emph{Quadrature collocation.\/} Energy minimization through quadrature rule with degree of precision 1.
    The number of cells and the blocks number of the Residual Network are varied.
   Top image:  For a given number of cells   $\norm{u_\theta - u_e}_{L^2(\varOmega)}$ is computed. 
   The minimum error $1.6 \cdot 10^{-3}$ is achieved for
 $3$ blocks and  $M=100$, i.e. 20000 triangles.
 Bottom image: Solid curves illustrate loss function values after the final epoch
  iteration (left vertical axis) and dashed curves the total execution time (right vertical axis).
  } 
  \label{fig_errors_autograd}
\end{figure}

\begin{figure}
\centering
 \includegraphics[width=1\linewidth]{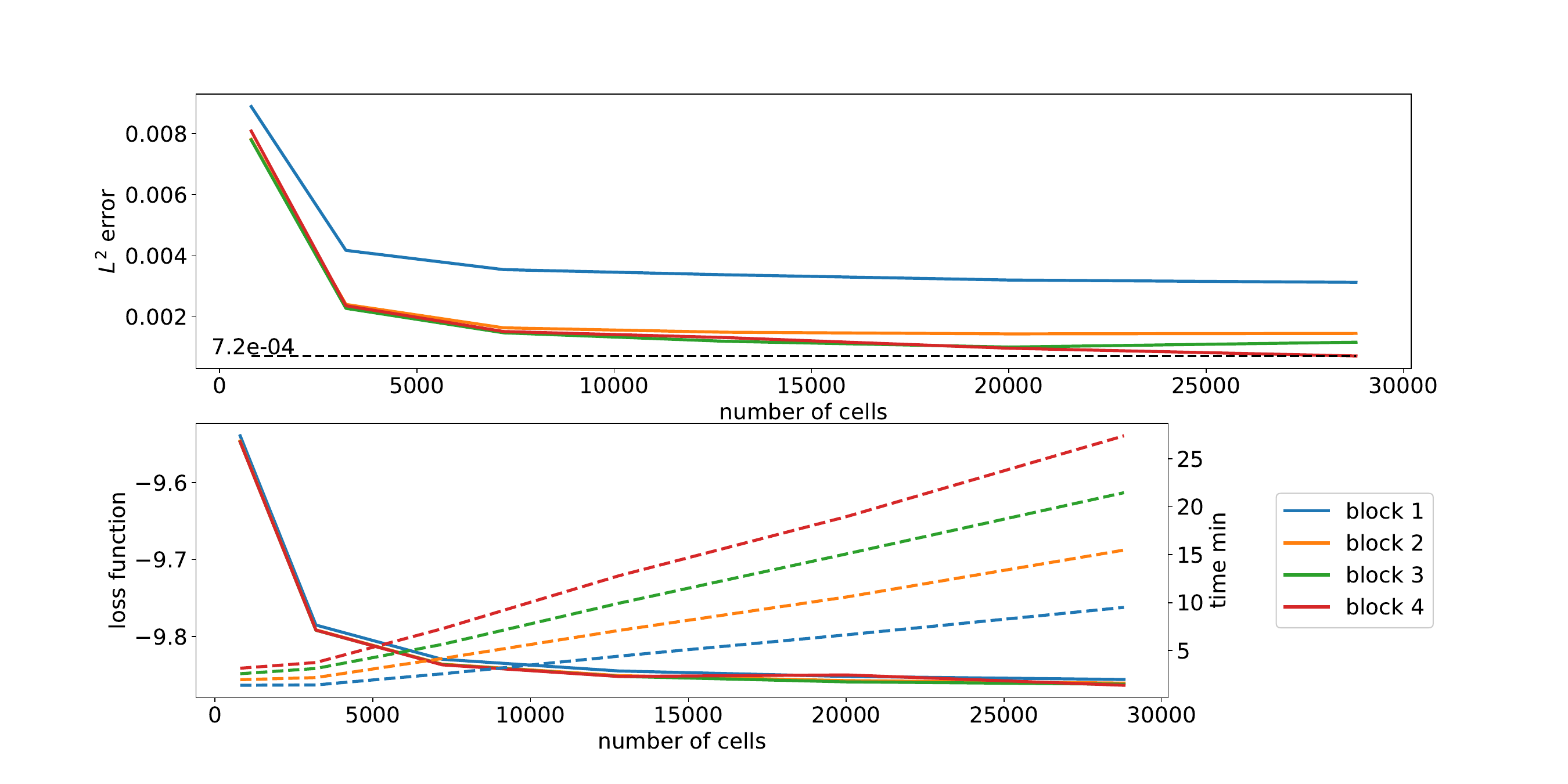}
\caption{\emph{Finite Element training.\/} Energy minimization through quadrature rule, with degree of precision 1, and finite element interpolation.
    The number of cells and the blocks number of the Residual Network are varied.
   Top image:  For a given number of cells   $\norm{u_\theta - u_e}_{L^2(\varOmega)}$ is computed. 
   The minimum error $7.2 \cdot 10^{-4}$ is achieved for
 $4$ blocks and  $M=120$, i.e. 28800 triangles.
 Bottom image: Solid curves illustrate loss function values after the final epoch
  iteration (left vertical axis) and dashed curves the total execution time (right vertical axis).
  }   \label{fig_errors_interpolant}
\end{figure}

%
%
\begin{figure}
\centering
 \includegraphics[width=0.35\linewidth]{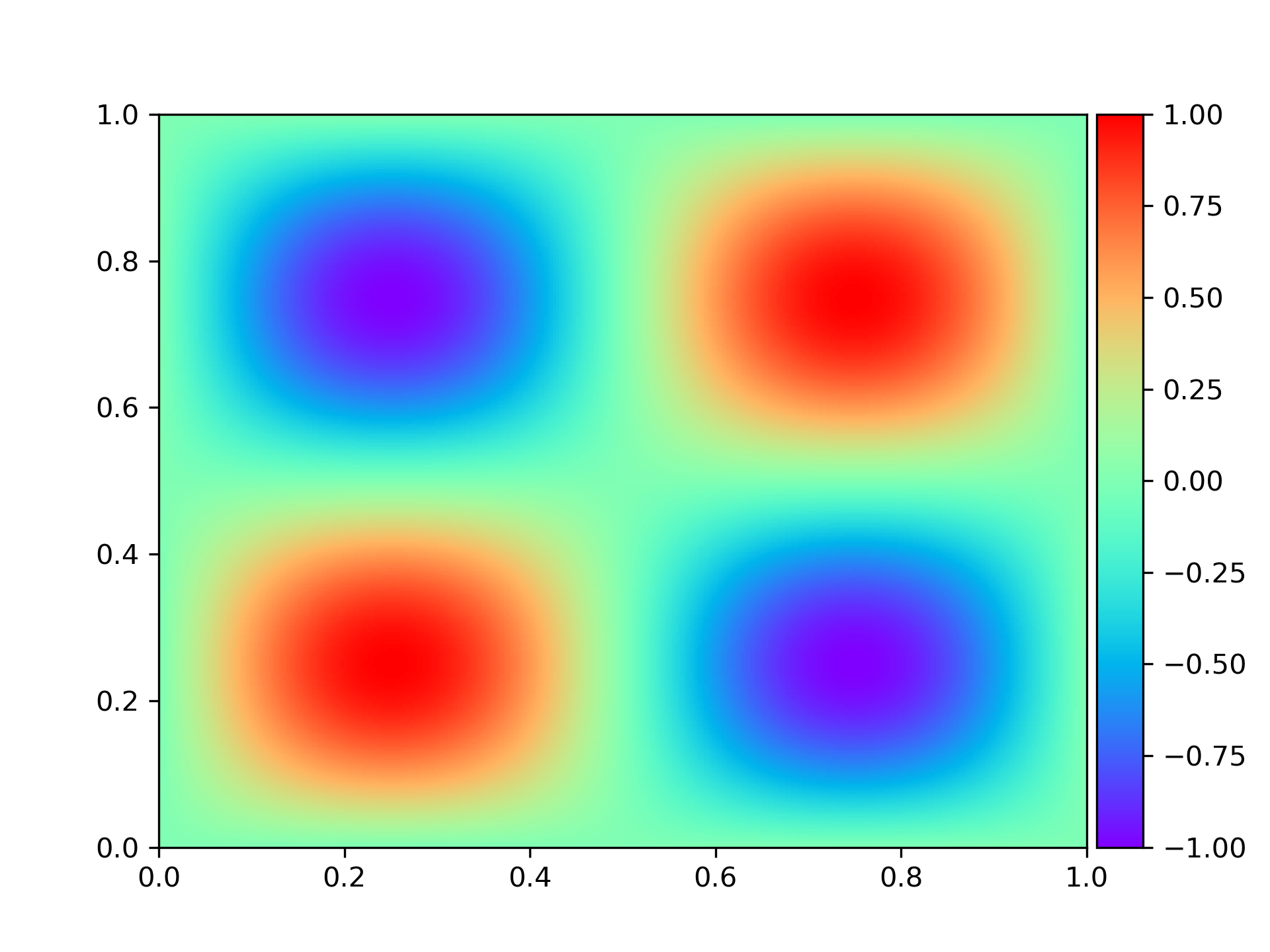}
  \includegraphics[width=0.35\linewidth]{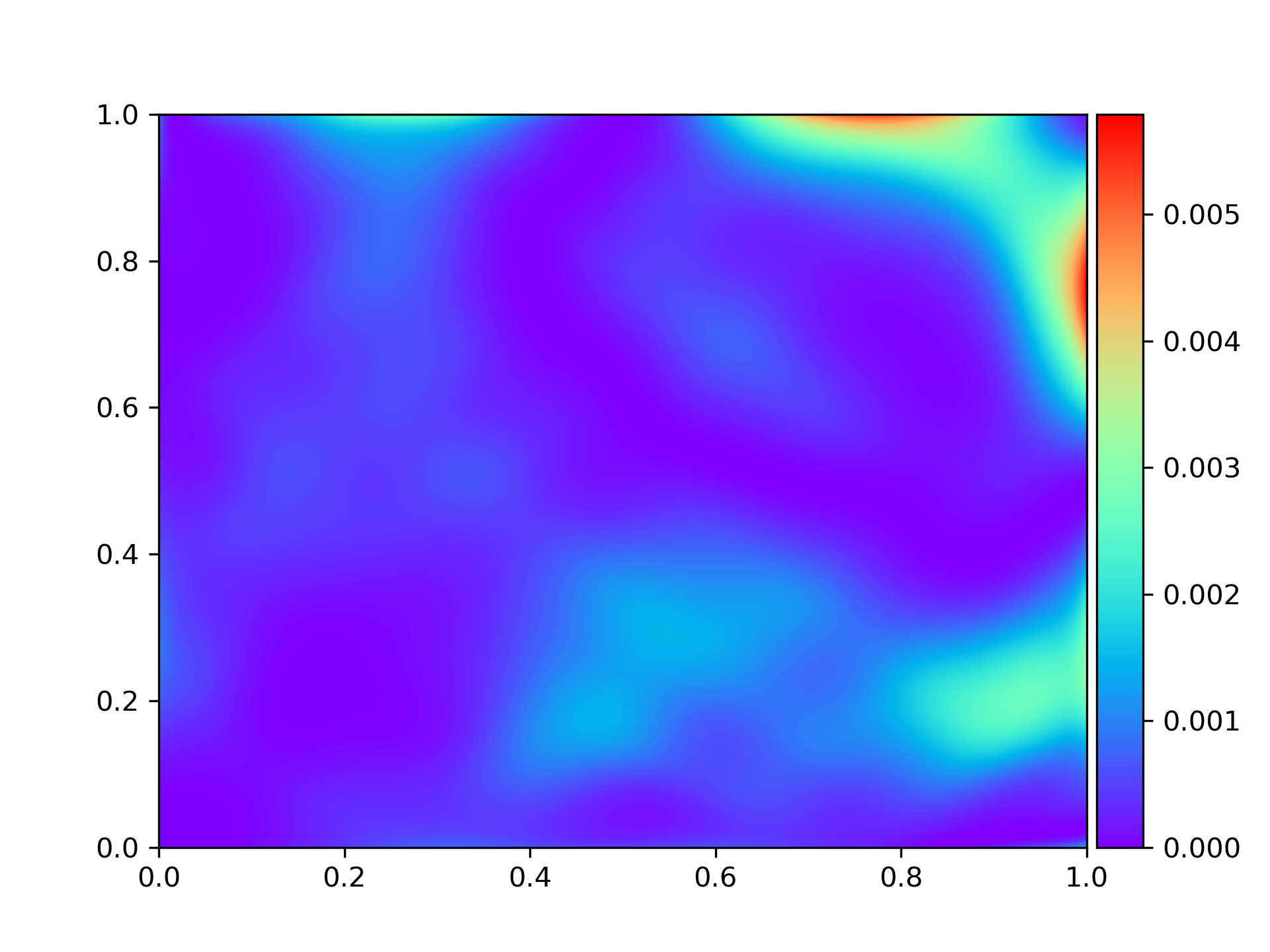}
      \includegraphics[width=0.35\linewidth]{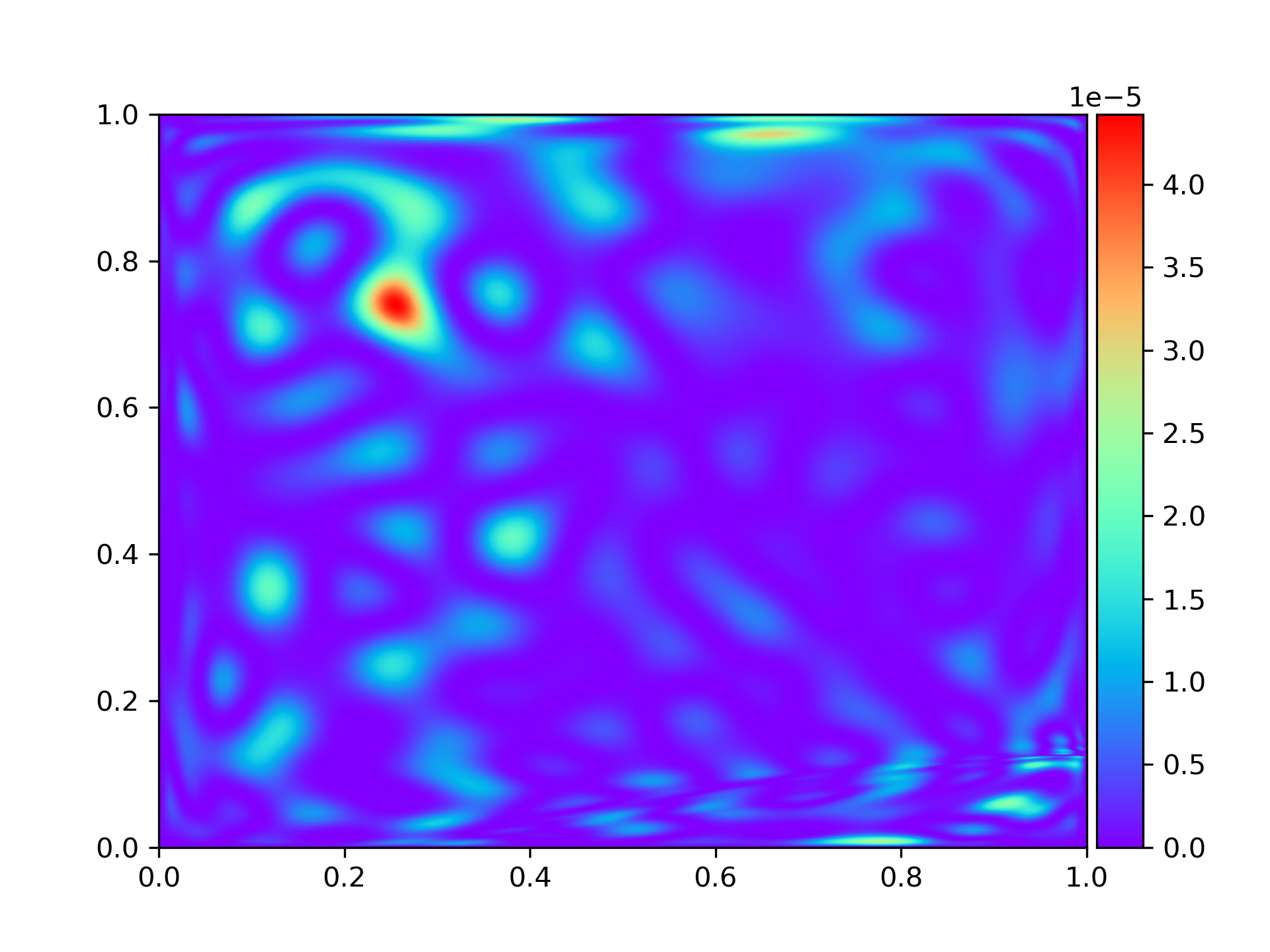}
      \includegraphics[width=0.35\linewidth]{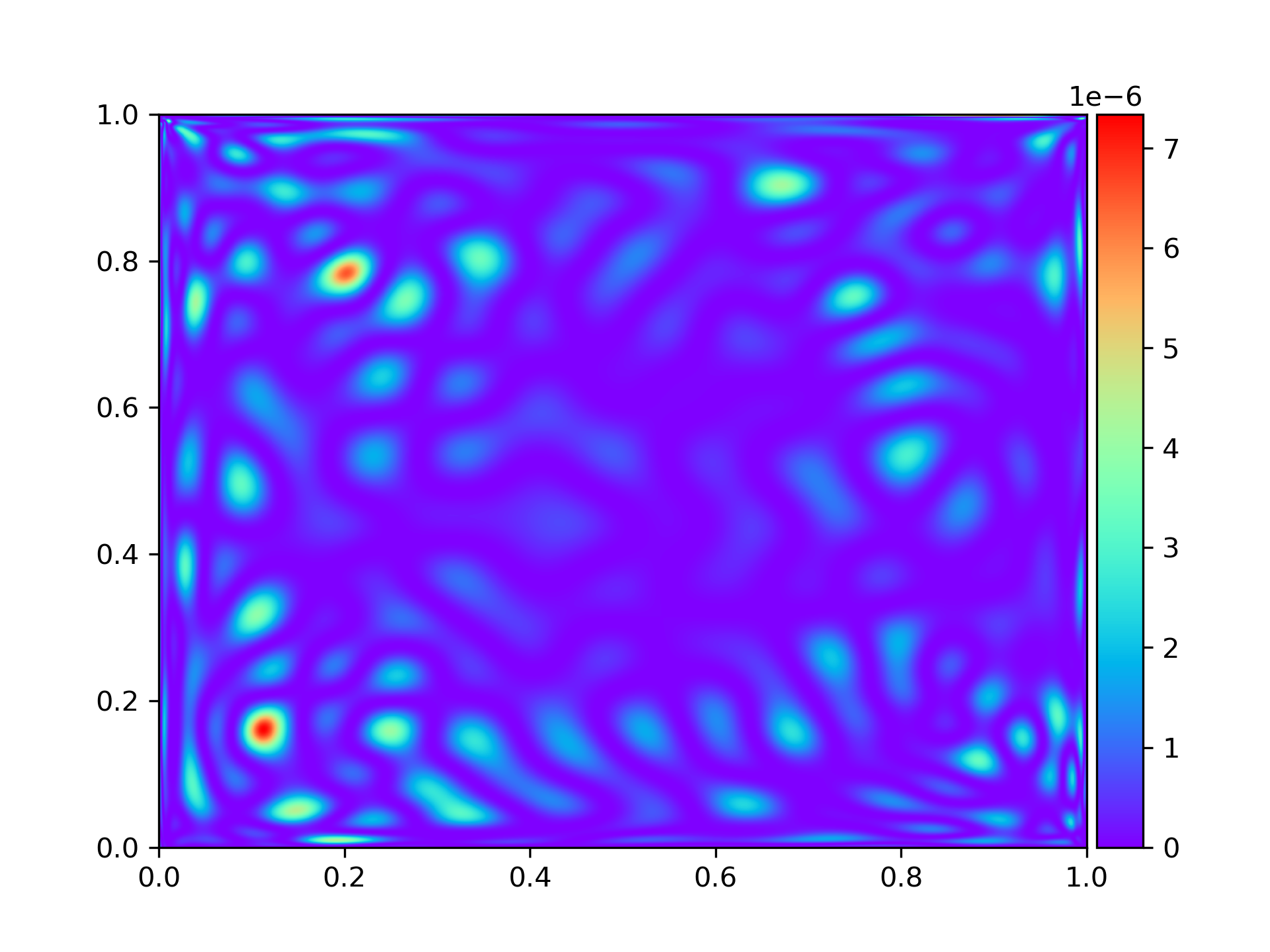}
  \caption{\emph{Pointwise errors.} Graph of the continuous minimizer $u_e$ (top left) and the pointwise squared difference with 
  its approximation, i.e. $|u_\theta - u_e|^2$. Top right: We have chosen $u_\theta$  with the best fitting for
   Monte-Carlo collocation  (blocks number$=1$, $N=80$ from Fig.~\ref{fig_errors_random}).  Bottom left: Best 
   $u_\theta$ from quadrature collocation  (blocks number$=3$, $M=100$, from Fig.~\ref{fig_errors_autograd}). 
   Bottom right: Best $u_\theta$ from training  with finite elements  (blocks number$=4$, $M=120$, from Fig.~\ref{fig_errors_interpolant}). 
  } 
  \label{fig:ue_and_diff}
\end{figure}

\begin{figure}
\centering
 \includegraphics[width=1\linewidth]{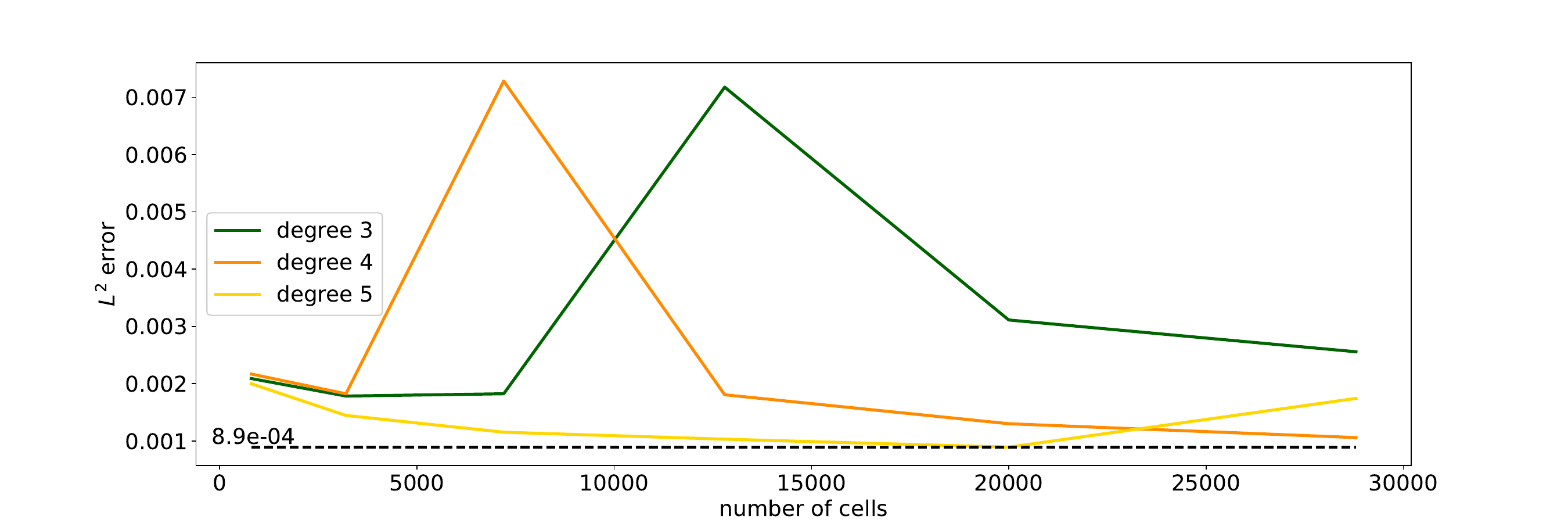}
  \caption{\emph{Training through quadrature -- higher order quadrature rules.\/} Fixing the number of blocks of the Residual Network to $4$, we perform training through quadrature
  varying the degree of precision from $3$ to $5$. The minimum value $\approx 8.9 \cdot 10^{-4}$ is attained 
  when the degree of precision is $5$  and $M=100$.}
  \label{fig:autograd_over_id}
\end{figure}

\vskip 0.5cm

\begin{figure}
\centering
 \includegraphics[width=1\linewidth]{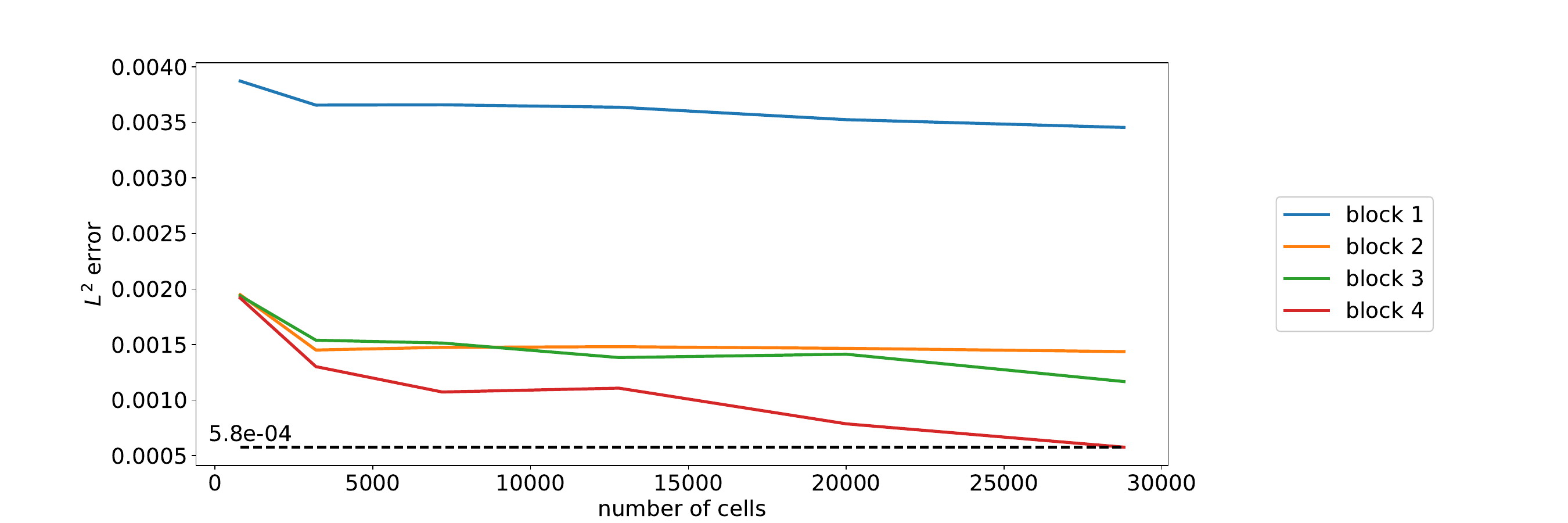}
  \caption{\emph{Finite Element training - $p=2.$\/} Energy minimisation through quadrature rule, with degree of precision 2, and finite element interpolation
  employing polynomials of second degree.
    The number of cells and the blocks number of the Residual Network are varied.
   Top image:  For a given number of cells   $\norm{u_\theta - u_e}_{L^2(\varOmega)}$ is computed. 
   The minimum error $5.8 \cdot 10^{-4}$ is achieved for
 $4$ blocks and  $M=120$, i.e. 28800 triangles.} 
  \label{fig:errors_interpolant2}
\end{figure}

\newpage

\bibliographystyle{amsplain}

\clearpage
\bibliography{references.bib}

\end{document}